\documentclass[12pt,a4paper,reqno]{amsart}
\usepackage[utf8]{inputenc}
\usepackage[american]{babel}
\usepackage{amsmath,amssymb,amsthm,amscd}
\usepackage{mathrsfs}
\usepackage[active]{srcltx}
\usepackage{pgf,tikz}
\usepackage{mathrsfs}
\usepackage{enumerate}
\usepackage{braket}
\usetikzlibrary{arrows}
\usepackage{dsfont}
\usepackage{paralist}
\usepackage{comment}
\usepackage[colorlinks,linkcolor={black},citecolor={black},urlcolor={black}]{hyperref}
\usepackage[margin=2.6cm,bmargin=3cm,tmargin=3cm]{geometry}

\newcommand{\1}{\mathds{1}}

\newcommand{\IC}{\mathbb{C}}
\newcommand{\IN}{\mathbb{N}}
\newcommand{\IR}{\mathbb{R}}
\renewcommand{\H}{\mathcal{H}}
\newcommand{\J}{\mathcal{J}}
\newcommand{\K}{\mathcal{K}}
\newcommand{\M}{\mathcal{M}}
\newcommand{\N}{\mathcal{N}}
\newcommand{\GE}{\mathrm{GE}}
\newcommand{\CGE}{\mathrm{CGE}}
\newcommand{\BE}{\mathrm{BE}}
\newcommand{\CBE}{\mathrm{CBE}}
\newcommand{\E}{\mathcal{E}}
\renewcommand{\L}{\mathcal{L}}

\newcommand{\Hess}{\operatorname{Hess}}
\newcommand{\Ent}{\operatorname{Ent}}
\newcommand{\norm}[1]{\lVert#1\rVert}
\newcommand{\abs}[1]{\lvert#1\rvert}
\newcommand{\id}{\mathrm{id}}
\newcommand{\W}{\mathcal{W}}
\newcommand{\I}{\mathcal{I}}

\newcommand{\cH}{\mathcal{H}}
\newcommand{\cM}{\mathcal{M}}
\newcommand{\cN}{\mathcal{N}}
\newcommand{\cP}{P}
\newcommand{\cI}{\mathcal{I}}
\newcommand{\cL}{\mathcal{L}}
\newcommand{\states}{\mathcal{S(M)}}
\renewcommand{\Re}{\operatorname{Re}}

\newcommand{\un}{\mathbf 1}

\theoremstyle{plain}
\newtheorem{theorem}{Theorem}[section]

\newtheorem{lemma}[theorem]{Lemma}
\newtheorem{proposition}[theorem]{Proposition}

\newtheorem{definition}[theorem]{Definition}

\theoremstyle{remark}
\newtheorem{remark}[theorem]{Remark}
\newtheorem{example}[theorem]{Example}

\newtheorem*{claim*}{Claim}
\newtheorem*{remark*}{Remark}
\newtheorem*{example*}{Example}
\newtheorem*{notation*}{Notation}
\numberwithin{equation}{section}

\makeatletter
\@namedef{subjclassname@2020}{\textup{2020} Mathematics Subject Classification}
\makeatother

\title[Curvature-dimension conditions for quantum Markov semigroups]{Curvature-dimension conditions for symmetric quantum Markov semigroups}
\author{Melchior Wirth}
\address[M. Wirth]{Institute of Science and Technology Austria (IST Austria),
Am Campus 1, 3400 Klosterneuburg, Austria}
\email{melchior.wirth@ist.ac.at}
\author{Haonan Zhang}
\address[H. Zhang]{Institute of Science and Technology Austria (IST Austria),
Am Campus 1, 3400 Klosterneuburg, Austria}
\email{haonan.zhang@ist.ac.at}

\begin{document}

\begin{abstract}
Following up on the recent work on lower Ricci curvature bounds for quantum systems, we introduce two noncommutative versions of curvature-dimension bounds for symmetric quantum Markov semigroups over matrix algebras. 
	%mention the work on nc lower Ricci curvature bounds?
	%This is the gradient estimate GE($K,N$) (and CGE($K,N$)) as a generalization of GE($K,\infty$) (and CGE($K,\infty$)) in our previous work \cite{WZ20}. 
	Under suitable such curvature-dimension conditions, we prove a family of dimension-dependent functional inequalities, a version of the Bonnet--Myers theorem and concavity of entropy power in the noncommutative setting. We also provide examples satisfying certain curvature-dimension conditions, including Schur multipliers over matrix algebras, Herz--Schur multipliers over group algebras and depolarizing semigroups.
	% Further extensions will also be discussed. 
\end{abstract}

\maketitle

%\tableofcontents

\section{Introduction}

Starting with the celebrated work by Lott--Villani \cite{LV09} and Sturm \cite{Stu06a,Stu06b}, recent years have seen a lot of research interest in extending the notion of Ricci curvature, or more precisely lower Ricci curvature bounds, beyond the realm of classical differential geometry to spaces with singularities \cite{AGS14a,AGS14b,AGS15,EKS15}, discrete spaces \cite{Maa11,EM12,Mie13} or even settings where there is no underlying space at all as for example in noncommutative geometry \cite{CM17,MM17,Wir18,CM20,JLL20,WZ20,DR20}.

Most of these approaches take as their starting point either the characterization of lower Ricci curvature bound in terms of convexity properties of the entropy on Wasserstein space \cite{vRS05} or in terms of Bakry--Émery's $\Gamma_2$-criterion \cite{BE85}, which derives from Bochner's formula, and in many settings, these two approaches yield equivalent or at least closely related notions of lower Ricci curvature bounds.

One of the reasons to seek to extend the notion of Ricci curvature beyond Riemannian manifolds is that lower Ricci curvature bounds have strong geometric consequences and are a powerful tool in proving functional inequalities. This motivated the investigation of lower Ricci curvature bounds in the noncommutative setting, or for quantum Markov semigroups.

From a positive noncommutative lower Ricci curvature bound in terms of the $\Gamma_2$-condition, Junge and Zeng \cite{JZ15a,JZ15b} derived a $L_p$-Poincaré-type inequality and transportation inequalities, and under such non-negative lower Ricci curvature bounds Junge and Mei proved $L_p$-boundedness of Riesz transform \cite{JM10}. Following Lott--Sturm--Villani, Carlen and Maas \cite{CM14,CM17,CM20} studied the noncommutative lower Ricci curvature bound via the geodesic semi-convexity of entropy by introducing a noncommutative analog of the $2$-Wasserstein metric. The similar approach was carried out by the first-named author in the infinite-dimensional setting in \cite{Wir18}. These two notions of lower Ricci curvature bounds are in general different, but they can both be characterized in terms of a gradient estimate \cite{Wir18,CM20,WZ20}. A stronger notion of lower Ricci curvature bound, which implies the bound in terms of $\Gamma_2$-condition and in terms of transportation, was introduced by Li, Junge and LaRacuente \cite{JLL20}. See also the further work of Li \cite{Li20}, and Brannan, Gao and Junge \cite{BGJ20a,BGJ20b}.

\smallskip

However, for many applications in geometric consequences such as the Bonnet--Myers theorem, and functional inequalities such as the concavity of entropy power, a lower bound on the Ricci curvature is not sufficient, but one needs an upper bound on the dimension as well. This leads to the curvature-dimension condition, whose noncommutative analog will be the main object of this article. As a finite-dimensional analog of lower Ricci curvature bounds, the curvature-dimension condition also admits various characterizations. Similar to the ``infinite-dimensional'' setting, two main approaches describing curvature-dimension conditions are $\Gamma_2$-criterion following Bakry--Émery and convexity properties of entropy on the $2$-Wasserstein space in the spirit of Lott--Sturm--Villani. For metric measure spaces, the equivalence of various characterizations on curvature-dimension conditions and their applications have been extensively studied beginning with \cite{EKS15}.% \textcolor{red}{(more references? I think the AGS papers cited from the very beginning are already representative. We can recite here if you want)}

While the notion of dimension is built into the definition of manifolds, it is not obvious in the extended settings and requires new definitions. The goal of this article is to provide such a definition of dimension (upper bounds) in the context of quantum Markov semigroups in a way that it fits well with the previously developed notions of lower Ricci curvature bounds in this framework. This definition allows us to prove interesting consequences on the geometry of the state space as well as some functional inequalities.

Furthermore, for quantum Markov semigroups satisfying an intertwining condition, which already appeared in \cite{CM17} and subsequent work, we provide an easily verifiable upper bound on the dimension, namely the number of partial derivatives in the Lindblad form of the generator. This sufficient condition enables us to prove the curvature-dimension condition in various concrete examples such as quantum Markov semigroups of Schur multipliers and semigroups generated by conditionally negative definite length functions on group algebras.

It should be mentioned that a notion of dimension for a quantum diffusion semigroup already appeared implicitly in the work of König and Smith on the quantum entropy power inequality \cite{KS14}. In particular, from their entropy power inequality one may also derive the concavity of entropy power for the associated quantum diffusion semigroup. See \cite{DPT18,HKV17,AB20} for more related work. % \textcolor{red}{(more relevant references here? I think here is enough. We can add some papers of de palma if you want)}.
This example fits conceptually well with our framework as it satisfies the intertwining condition and the dimension in the entropy power considered there is the number of partial derivatives in the Lindblad form of the generator, although the semigroup acts on an infinite-dimensional algebra and is therefore not covered by our finite-dimensional setting. Here we consider the concavity of the entropy power for arbitrary symmetric quantum Markov semigroups over matrix algebras.

In this paper we will focus on two noncommutative analogues of curvature-dimension conditions: the Bakry--Émery curvature dimension condition BE($K,N$), formulated via the $\Gamma_2$-condition, and the gradient estimate GE($K,N$), which is in the spirit of Lott--Sturm--Villani when the reference operator mean is chosen to be the logarithmic mean. They are generalizations of ``infinite-dimensional'' notions BE($K,\infty$) and GE($K,\infty$) in previous work, but let us address one difference in the ``finite-dimensional'' setting, i.e. $N<\infty$. As we mentioned above, in the ``infinite-dimensional'' case, i.e. $N=\infty$, GE($K,\infty$) recovers BE($K,\infty$) if the operator mean is the left/right trivial mean. However, this is not the case when $N<\infty$; BE($K,N$) is stronger than GE($K,N$) for the left/right trivial mean. 

\smallskip

This article is organized as follows. Section \ref{sec:QMS} collects preliminaries about quantum Markov semigroups and noncommutative differential calculus that are needed for this paper. In Section \ref{sec:BE} we study the noncommutative Bakry--Émery curvature-dimension condition BE($K,N$), its applications and the complete version. In Section \ref{sec:CD GE} we investigate the noncommutative gradient estimate GE($K,N$) for arbitrary operator means, give an equivalent formulation in the spirit of the $\Gamma_2$-criterion, and also introduce their complete form. Section \ref{sec:geodesic convexity} is devoted to the gradient estimate GE($K,N$), its connection to the geodesic $(K,N)$-convexity of the (relative) entropy and applications to dimension-dependent functional inequalities. In Section \ref{sec:examples} we give some examples of quantum Markov semigroups for which our main results apply. In Section \ref{sec:conclusion} we discuss how to extend the theory from this article to quantum Markov semigroups  that are not necessarily tracially symmetric and explain the main challenge in this case.

\subsection*{Acknowledgments} H.Z. is supported by the European Union's Horizon 2020 research and innovation programme under the Marie Sk\l odowska-Curie grant agreement No. 754411. M.W. acknowledges support from the European Research Council (ERC) under the European Union’s Horizon 2020 research and innovation programme (grant agreement No 716117) and from the Austrian Science Fund (FWF) through grant number F65. Both authors would like to thank Jan Maas for fruitful discussions and helpful comments.

\section{Quantum Markov semigroups and noncommutative differential calculus}
\label{sec:QMS}

In this section we give some background material on quantum Markov semigroups, their generators, first-order differential calculus and operator means.

\subsection{Quantum Markov semigroups}

Throughout we fix a finite-dimensional von Neumann algebra $\M$ with a faithful tracial state $\tau$. By the representation theory of finite-dimensional $C^\ast$-algebras, $\M$ is of the form $\bigoplus_{j=1}^n M_{k_j}(\IC)$ and $\tau=\bigoplus_{j=1}^n \alpha_j \mathrm{tr}_{M_{k_j}(\IC)}$ with $\alpha_j\geq 0$, $\sum_{j=1}^n \alpha_j k_j=1$. Here $M_n(\IC)$ denotes the full $n$-by-$n$ matrix algebra and $\mathrm{tr}_{M_{n}(\IC)}$ is the usual trace over $M_{n}(\IC)$.

\smallskip

Denote by $\M_+$ the set of positive semi-definite matrices in $\M$. A density matrix is a positive element $\rho\in \M$ with $\tau(\rho)=1$. The set of all density matrices is denoted by $\states$ and the set of all invertible density matrices by $\mathcal{S}_+(\M)$. We write $L_2(\M,\tau)$ for the Hilbert space obtained by equipping $\M$ with the inner product
\begin{equation*}
\langle \cdot,\cdot\rangle\colon \M\times\M\to\IC,\,(x,y)\mapsto \tau(x^\ast y).
\end{equation*}
The adjoint of a linear operator $T\colon \M\to\M$ with respect to this inner product is denoted by $T^\dagger$.

\smallskip

A family $(P_t)_{t\geq 0}$ of linear operators on $\M$ is called a \emph{quantum Markov semigroup} if 
\begin{enumerate}[(a)]
\item $P_0=\id_\M$, $P_{s+t}=P_sP_t$ for $s,t\geq 0$,
\item $P_t$ is completely positive for every $t\geq 0$,
\item $P_t \un=\un$ for every $t\geq 0$,
\item $t\mapsto P_t$ is continuous.
\end{enumerate}
The generator of $(P_t)$ is
\begin{equation*}
\L\colon \M\to\M,\,\L x=\lim_{t\searrow 0}\frac 1 t(x-P_t(x)).
\end{equation*}
It is the unique linear operator on $\M$ such that $P_t=e^{-t\L}$. Let us remark that sign conventions differ and sometimes $-\L$ is called the generator of $(P_t)$.
%
%Recall that the (positive) generator $\cL$ of a symmetric QMS over $\cM=B(\cH)$ with $\dim \cH<\infty$ admits the Lindblad form:
%\begin{equation}\label{eq:lindblad}
%\cL=\sum_{j=1}^{d}\partial_{j}^\dagger\partial_j,
%\end{equation}
%where $\partial_j=[V_j,\cdot]$ and $\{V_j\}=\{V_j^\ast\}$. We denote by $\partial:\cM\to \hat{\cM}$ the operator 
%\begin{equation*}
%\partial x:=(\partial_1 x,\dots, \partial_d x).
%\end{equation*}
% Then $\cL=\partial^\dagger\partial$. We denote by $\E$ the energy:
% \begin{equation*}
% \E(A,B)=\langle \cL(A),B\rangle=\langle \partial A,\partial B\rangle.
% \end{equation*}

\smallskip

Let $\sigma\in \mathcal{S}_+(\M)$. The quantum Markov semigroup $(P_t)$ is said to satisfy the \emph{$\sigma$-detailed balance condition} ($\sigma$-DBC) if
\begin{equation*}
\tau(P_t(x)y\sigma)=\tau(x P_t(y)\sigma)
\end{equation*}
for $x,y\in \M$ and $t\geq 0$. In the special case $\sigma=\un$ we say that $(P_t)$ is \emph{tracially symmetric} or \emph{symmetric}, and denote
$$\E(a,b):=\langle a,\L b \rangle.$$ 
A tracially symmetric quantum Markov semigroup $(P_t)$ is \emph{ergodic} if $\un$ is the unique invariant state of $(P_t)$.

Although it is not necessary to \emph{formulate} the curvature-dimension conditions, we will deal exclusively with tracially symmetric quantum Markov semigroups since all examples where we can verify the conditions fall into that class. As a special case of Alicki's theorem \cite[Theorem 3]{Alicki76} (see also \cite[Theorem 3.1]{CM17}) the generator $\L$ of a tracially symmetric quantum Markov semigroup on $\M=M_n(\IC)$ is of the form
\begin{equation*}
\L=\sum_{j\in \J}\partial_j^\dagger\partial_j,
\end{equation*}
where $\J$ is a finite index set, $\partial_j=[v_j,\cdot\,]$ for some $v_j\in \M$, and for every $j\in \J$ there exists a unique $j^\ast\in \J$ such that $v_j^\ast=v_{j^\ast}$. We call the operators $\partial_j$ partial derivatives. Using the derivation operator $\partial:=(\partial_j)_{j\in\J}:\M\to  \hat{\cM}:=\oplus_{j\in\J}\cM$, we may also write $\L=\partial^\dagger\partial$.

\subsection{Noncommutative differential calculus and operator means}
Let us shortly recall the definition and some basic properties of operator means. Let $\H$ be an infinite-dimensional Hilbert space. A map $\Lambda\colon B(\cH)_+\times B(\cH)_+\to B(\cH)_+$ is called an \emph{operator connection} if it satisfies the following properties.
\begin{enumerate}[(a)]
	\item \emph{monotonicity}: if $A\leq C$ and $B\leq D$, then $\Lambda(A,B)\leq\Lambda(C,D)$,
	\item \emph{transformer inequality}: $C \Lambda(A,B)C\leq \Lambda(C A C,C B C)$ for any $A,B,C\in B(\cH)_+$,
	\item \emph{continuity}: $A_n\searrow A$ and $B_n\searrow B$ imply $\Lambda(A_n,B_n)\searrow \Lambda(A,B)$.
\end{enumerate}
An operator connection $\Lambda$ is called an \emph{operator mean} if it additionally satisfies
\begin{enumerate}[(a)]
\setcounter{enumi}{3}
	\item $\Lambda(\id_{\cH},\id_{\cH})=\id_{\cH}$.
\end{enumerate}
Here by $A_n\searrow A$ we mean $A_1\ge A_2\ge \cdots$ and  $A_n$ converges strongly to $A$.
% From (b), any operator mean $\Lambda$ is \emph{positively homogeneous}:
%$$\Lambda(\lambda A,\lambda B)=\lambda\Lambda (A,B),~~\lambda >0,A,B\in B(\cH)_+.$$
The operator connection $\Lambda$ is \emph{symmetric} if $\Lambda(A,B)=\Lambda(B,A)$ for all $A,B\in B(\cH)_+$.

\begin{lemma}\label{lem:mean}
	Let $\Lambda$ be an operator connection. Then for $\lambda\ge 0$, $A,B,C,D\in B(\cH)_+$ and unitary $U\in B(\cH)$, we have
	\begin{enumerate}[(a)]
		\item positive homogeneity: $\Lambda(\lambda A,\lambda B)=\lambda\Lambda (A,B)$,
		\item concavity: $\Lambda (A,C)+\Lambda (B,D)\le \Lambda (A+B,C+D)$,
		\item unitary invariance: $\Lambda(U^\ast A U,U^\ast B U)=U^\ast\Lambda(A,B)U$.
	\end{enumerate}
If $\Lambda$ is an operator mean, then additionally
\begin{enumerate}[(a)]
\setcounter{enumi}{3}
 		\item $\Lambda(A,A)=A$.
\end{enumerate}	
\end{lemma}

\begin{proof}
	See equations (II$_0$), (2.1), Theorem 3.3 and Theorem 3.5 in \cite{KA80}.
\end{proof}
While operator connections are initially only defined for bounded operators on an infinite-dimensional Hilbert space, one can easily extend this definition to operators on finite-dimensional Hilbert spaces as follows. If $\Lambda$ is an operator connection, $\H$ is a finite-dimensional Hilbert space and $A,B\in B(\H)_+$, then one can define $\Lambda(A,B)$ as $V^\ast\Lambda(VAV^\ast,VBV^\ast)V$, where $V$ is an isometric embedding of $\H$ into an infinite-dimensional Hilbert space. The unitary invariance from the previous lemma ensures that this definition does not depend on the choice of the embedding $V$.

\smallskip

Let $L(\rho)$ and $R(\rho)$ be the left and right multiplication operators, respectively, and fix an operator mean $\Lambda$. For $\rho\in\cM_+$ we define
\begin{equation*}
\hat \rho=\Lambda(L(\rho), R(\rho)).
\end{equation*}
Of particular interest for us are the cases when $\Lambda$ is the \emph{logarithmic mean}
\begin{equation*}
\Lambda_{\text{log}}(L(\rho), R(\rho))=\int_{0}^{1}L(\rho)^s R(\rho)^{1-s}\,ds,
\end{equation*}
or the \emph{left/right trivial mean}
\begin{equation*}
\Lambda_{\text{left}}(L(\rho), R(\rho))=L(\rho),~~\Lambda_{\text{right}}(L(\rho), R(\rho))=R(\rho).
\end{equation*}
%\emph{arithmetic mean}
%\begin{equation*}
%\Lambda_{\text{ari}}(L(\rho), R(\rho))=\frac{L(\rho)+R(\rho)}{2}.
%\end{equation*}
With $\Lambda=\Lambda_{\text{log}}$ being the logarithmic mean, we have the chain rule identity for $\log$ (see \cite[Lemma 5.5]{CM17} for a proof):
\begin{equation*}%\label{eq:chain rule for log}
\partial\rho=\hat{\rho}\partial\log\rho=\int_{0}^1\rho^s(\partial \log\rho) \rho^{1-s}ds.
\end{equation*}
Here and in what follows, we use the notation 
$$\hat{\rho}(x_1,\dots,x_n):=(\hat{\rho} x_1,\dots, \hat{\rho}x_n).$$

\section{\texorpdfstring{Bakry--Émery curvature-dimension condition BE($K,N$)}{Bakry-Émery curvature-dimension condition BE(K,N)}}% and curvature-dimension condition CD($K,N$)
% and gradient estimate condition GE($K,N$)

\label{sec:BE}
%This section is devoted to the noncommutative analogue of the Bakry--Émery curvature-dimension condition BE($K,N$) and the noncommutative curvature-dimension condition CD($K,N$) (or equivalently GE($K,N$)) in the spirit of Lott--Sturm--Villani. The former is formulated through the $\Gamma$-calculus, and the latter (when the operator mean is the logarithmetic mean) is defined via the geodesic $(K,N)$-convexity of the entropy with respect to some noncommutative transport metric. We remark that the latter approach is very flexible in the sense that one may choose different operator means to derive different transport metric (or to interpret different gradient estimates). Moreover, in the ``infinite-dimensional'' case, i.e. $N=\infty$, CD($K,\infty$) recovers BE($K,\infty$) when the operator mean is chosen to be the left/right trivial mean. However, this is not the case when $N<\infty$. More discussions can be found in ?.

%We start with Bakry--Émery curvature-dimension condition BE($K,N$), which is easier.

This section is devoted to the noncommutative analog of the Bakry--Émery curvature-dimension condition BE($K,N$) defined by the $\Gamma_2$-criterion. After giving the definition, we will show that it is satisfied for certain generators in Lindblad form, where the dimension parameter $N$ is given by the number of partial derivatives. We will then prove that $\BE(K,N)$ implies an improved Poincaré inequality. In the final part of this section we study a complete version of $\BE(K,N)$, called $\CBE(K,N)$, and show that it has the expected tensorization properties.

\subsection{\texorpdfstring{Bakry--Émery curvature-dimension condition BE($K,N$)}{Bakry-Émery curvature-dimension condition BE(K,N)}}
Let $(P_t)$ be a quantum Markov semigroup on $\M$ with generator $\L$. The associated \emph{carré du champ operator} $\Gamma$ is defined as 
\begin{equation*}
\Gamma(a,b):=\frac{1}{2}\left( a^\ast \cL b+(\cL a)^\ast b-\cL(a^\ast b)\right),
\end{equation*}
and the \emph{iterated carré du champ operator} $\Gamma_2$ is defined as
\begin{equation*}
\Gamma_2(a,b):=\frac{1}{2}\left( \Gamma(a,\cL b) +\Gamma(\cL a,b)-\cL\Gamma(a, b)\right).
\end{equation*}
As usual, we write $\Gamma(a)$ for $\Gamma(a,a)$ and $\Gamma_2(a)$ for $\Gamma_2(a,a)$.

\begin{proposition}\label{prop:char_BE}
Let $K\in\IR$ and $N\in (0,\infty]$. For a quantum Markov semigroup $(P_t)$ over $\M$ with generator $\L$, the following are equivalent:
\begin{enumerate}[(a)]
	
	\item for any $t\ge 0$ and any $a\in \M$:
	\begin{equation*}
	\Gamma(P_t a)\le e^{-2Kt}P_t\Gamma(a)-\frac{1-e^{-2Kt}}{KN}|\cL P_t a|^2,
	\end{equation*}
	\item for any $a\in \M$:
	\begin{equation*}
	\Gamma_2(a)\ge K\Gamma(a)+\frac{1}{N}|\cL a|^2.
	\end{equation*}

\end{enumerate}	
If this is the case, we say the semigroup $(P_t)$ satisfies Bakry--Émery curvature-dimension condition $\BE(K,N)$.
\end{proposition}

\begin{proof}
	The proof is essentially based on the following identities: For $s\in [0,t]$,
	\begin{equation*}
	\frac{d}{ds}P_s((P_{t-s}a)^\ast (P_{t-s}a))=2P_s \Gamma(P_{t-s}a),
	\end{equation*}
	and 
	\begin{equation*}
	\frac{d}{ds}P_s \Gamma(P_{t-s}a)=2P_s \Gamma_2(P_{t-s}a),
	\end{equation*}
	which follow by direct computations. 
	To prove $(a)\implies (b)$, we set 
	\begin{equation*}
	\phi(t):=e^{-2Kt}P_t\Gamma(a)-\Gamma(P_t a)-\frac{1-e^{-2Kt}}{KN}|\cL P_t a|^2.
	\end{equation*}
	Since $\phi(t)\ge 0$ for all $t\ge 0$ and $\phi(0)=0$, we have $\phi'(0)\ge 0$, which is nothing but (b).
	
	To show $(b)\implies (a)$, we put for any $t>0$: 
	\begin{equation*}
	\varphi(s):=e^{-2Ks}P_s\Gamma(P_{t-s}a),~~s\in [0,t].
	\end{equation*}
	Then by assumption and Kadison-Schwarz inequality,
	\begin{equation*}
	\varphi'(s)
	%=e^{-2Ks}\left(\frac{d}{ds}P_s \Gamma(P_{t-s}a)-2KP_s \Gamma(P_{t-s}a)\right)
	=2e^{-2Ks}P_s\left(\Gamma_2(P_{t-s}a)-K\Gamma(P_{t-s}a)\right)
	\ge \frac{2e^{-2Ks}}{N}P_s\left(|\cL P_{t-s}a|^2\right)
	\ge \frac{2e^{-2Ks}}{N}|\cL P_{t}a|^2.
	\end{equation*}
	So 
	\begin{equation*}
	\varphi(t)-\varphi(0)=\int_{0}^{t}\varphi'(s)ds
	\ge \frac{2}{N}\int_{0}^{t}e^{-2Ks}ds|\cL P_{t}a|^2
	=\frac{1-e^{-2Kt}}{KN}|\cL P_{t}a|^2,
	\end{equation*}
	which proves (a).
\end{proof}

\begin{remark}\label{rk:implicit constant_BE}
	From the proof one can see that the function 
	$$t\mapsto \frac{1-e^{-2Kt}}{KN},$$ 
	in (a) can be replaced by any $f$ such that $f(0)=0$ and $f'(0)=2/N$.
\end{remark}

\begin{remark}
	The notion $\BE(K,N)$ is clearly consistent: If $(P_t)$ satisfies $\BE(K,N)$, then it also satisfies $\BE(K',N')$ for all $K'\le K$ and $N'\ge N$.
\end{remark}

\begin{remark}
While all our examples of quantum Markov semigroups satisfying $\BE$ are tracially symmetric, let us point out that this is not necessary for the definition nor for the results in the rest of this section with the exception of Proposition \ref{prop:BE_intertwining}. See also the discussion in Section \ref{sec:conclusion}.
\end{remark}

We shall give a sufficient condition for Bakry--Émery curvature-dimension condition BE($K,N$). Before that we need a simple inequality. 

\begin{lemma}\label{lem:simple ineq}
	For any $a_{j},1\le j\le d$, in a C*-algebra, we have 
	\begin{equation*}%\label{ineq:simple ineq}
	\sum_{j=1}^{d}|a_{j}|^2\ge \frac{1}{d}\left|\sum_{j=1}^{d}a_{j}\right|^2.
	\end{equation*}
\end{lemma}

\begin{proof}
	In fact, 
	\begin{equation*}
	d\sum_{j=1}^{d}|a_{j}|^2-\left|\sum_{j=1}^{d}a_{j}\right|^2
	=\frac{1}{2}\sum_{j,k=1}^{d}\left(a_j^\ast a_j+a_k^\ast a_k-a_j^\ast a_k-a_k^\ast a_j\right)
	=\frac{1}{2}\sum_{j,k=1}^{d}|a_j-a_k|^2
	\ge 0.\qedhere
	\end{equation*}
\end{proof}

\begin{definition}
	Suppose that $\mathcal{L}$ is the generator of the tracially symmetric quantum Markov semigroup $(P_t)$ with the Lindblad form:
	\begin{equation*}\label{eq:lindblad}
	\mathcal{L}=\sum_{j=1}^{d}\partial_j^\dagger \partial_j,\tag{LB}
	\end{equation*}
	where $\partial_j(\cdot)=[v_j,\cdot]$ with the adjoint being $\partial_j^\dagger(\cdot)=[v_j^\ast,\cdot]$, and $\{v_j\}=\{v_j^\ast\}$. Then we say $(P_t)$ satisfies the $K$-intertwining condition for some $K\in\mathbb{R}$ if 
	\begin{equation*}%\label{eq:intertwining-semigroup}
	\partial_j P_t=e^{-K t}P_t\partial_j,~~1\le j\le d,
	\end{equation*}
	or equivalently
	\begin{equation*}%\label{eq:intertwining-generator}
	\partial_j \cL=\cL\partial_j+K\partial_j,~~1\le j\le d.
	\end{equation*}
\end{definition}

\begin{proposition}\label{prop:BE_intertwining}
	Suppose that the generator $\cL$ of the tracially symmetric quantum Markov semigroup $(P_t)$ admits the Lindblad form \eqref{eq:lindblad}. Then for any $a$,
	\begin{equation}\label{eq:gamma2 bochner}
	\Gamma_2(a)=\Re\sum_{j=1}^{d}(\partial_j\mathcal{L}a-\mathcal{L}\partial_j a)^\ast\partial_j a+\sum_{j,k=1}^{d}|\partial_k^\dagger\partial_j a|^2.
	\end{equation}
	If $(P_t)$ satisfies the $K$-intertwining condition for $K\in\mathbb{R}$, then $(P_t)$ satisfies $\BE(K,d)$.
\end{proposition}

\begin{proof}
	 Note that 
	\begin{equation*}%\label{eq:partial-partial dagger}
	(\partial_j a)^\ast=a^\ast v_j^\ast-v_j^\ast a^\ast=-\partial_j^\dagger(a^\ast).
	\end{equation*}
	%and (since $\{V_j\}=\{V_j^\ast\}$ and $\{\partial_j\}=\{\partial_j^\dagger\}$)
	%\begin{equation*}
	%\mathcal{L}=\sum_{j}\partial_j^\dagger \partial_j=\sum_{j}\partial_j \partial_j^\dagger.
	%\end{equation*}
	This, together with the Leibniz rule for $\partial_j$'s (so also $\partial_j^\dagger$'s),  and the fact that $\{\partial_j\}=\{\partial_j^\dagger\}$, yields
	\begin{align*}
	\cL(a^\ast b)
	%=\sum_{j=1}^{d}\partial_j^\dagger \partial_j(a^\ast b)
	&=\sum_{j=1}^{d}(\partial_j^\dagger \partial_ja^\ast)b+a\partial_j^\dagger \partial_j b+(\partial_j^\dagger a)(\partial_j b)+(\partial_j a)(\partial_j^\dagger b)\\
	&=(\cL a)^\ast b+a^\ast\cL b+\sum_{j=1}^{d}(\partial_j^\dagger a^\ast)(\partial_j b)+(\partial_j a^\ast)(\partial_j^\dagger b)\\
	&=(\cL a)^\ast b+a^\ast\cL b-\sum_{j=1}^{d}\left((\partial_j a)^\ast(\partial_j b)+(\partial_j^\dagger a)^\ast(\partial_j^\dagger b)\right)\\
	&=(\cL a)^\ast b+a^\ast\cL b-2\sum_{j=1}^{d}(\partial_j a)^\ast(\partial_j b).
	\end{align*}
	So by definition, the carré du champ operator is given by:
	\begin{equation}\label{eq:gamma interms of partial}
	\Gamma(a,b)
	=\frac{1}{2}\left(a^\ast \cL b+(\cL a)^\ast b-\cL(a^\ast b) \right)
	=\sum_{j=1}^{d}(\partial_j a)^\ast \partial_j b.
	\end{equation}
	The above computations yield
	\begin{align*}
	\Gamma(a,\mathcal{L}(a))+\Gamma(\mathcal{L}(a),a)
	=\sum_{j=1}^{d}(\partial_j \cL a)^\ast \partial_j a+(\partial_j a)^\ast \partial_j \cL a
	=2\Re \sum_{j=1}^{d}(\partial_j \cL a)^\ast \partial_j a,
	\end{align*}
	and
	\begin{align*}
	\cL\Gamma(a)
	&=\sum_{j=1}^{d}\cL\left((\partial_j a)^\ast \partial_j a\right)\\
	&=\sum_{j=1}^{d}\left((\cL\partial_j a)^\ast \partial_j a+(\partial_j a)^\ast \cL \partial_j a-2\sum_{k=1}^{d}(\partial_k\partial_j a)^\ast \partial_k\partial_j a\right)\\
	&=2\Re\sum_{j=1}^{d}(\cL\partial_j a)^\ast \partial_j a
	%+(\partial_j a)^\ast \cL \partial_j a\right)
	-2\sum_{j,k=1}^{d}|\partial_k\partial_j a|^2.
	\end{align*}
	Thus
	\begin{align*}
	\Gamma_2(a)	=&\frac{1}{2}\left(\Gamma(a,\mathcal{L}(a))+\Gamma(\mathcal{L}(a),a)-\mathcal{L}\Gamma(a)\right)\\
	=&\Re \sum_{j=1}^{d}(\partial_j \cL a)^\ast \partial_j a-\Re \sum_{j=1}^{d}(\cL\partial_j a)^\ast \partial_j a+\sum_{j,k=1}^{d}|\partial_k\partial_j a|^2\\
	=&\Re \sum_{j=1}^{d}(\partial_j \cL a-\cL\partial_j a)^\ast \partial_j a+\sum_{j,k=1}^{d}|\partial_k^\dagger\partial_j a|^2,
	\end{align*}
	where in the last equality we used again the fact that $\{\partial_j\}=\{\partial_j^\dagger\}$. This proves \eqref{eq:gamma2 bochner}. 
	If $(P_t)$ satisfies the $K$-intertwining condition, then
	\begin{equation*}
	\Re\sum_{j=1}^d(\partial_j \mathcal{L}a-\mathcal{L}\partial_j a)^\ast(\partial_j a)=K\sum_{j=1}^d (\partial_j a)^\ast(\partial_j a)=K\Gamma(a).
	\end{equation*}
	Moreover, by Lemma \ref{lem:simple ineq} we get
	\begin{equation*}
	\sum_{j,k=1}^d \abs{\partial_k^\dagger \partial_j a}^2
	\geq \sum_{j=1}^d \abs{\partial_j^\dagger \partial_j a}^2
	\ge \frac 1 d \left\lvert\sum_{j=1}^d \partial_j^\dagger\partial_j a\right\rvert^2
	=\frac 1 d \abs{\mathcal{L}a}^2.
	\end{equation*}
	Therefore $(P_t)$ satisfies $\BE(K,d)$:
	\begin{equation*}
	\Gamma_2(a)=\Re \sum_{j=1}^{d}(\partial_j \cL a-\cL\partial_j a)^\ast \partial_j a+\sum_{j,k=1}^{d}|\partial_k^\dagger\partial_j a|^2
	\geq K\Gamma(a)+\frac 1 d \abs{\mathcal{L}a}^2.\qedhere
	\end{equation*}
\end{proof}

\subsection{Applications}

In this subsection we present two applications of the Bakry--Émery curvature-dimension condition, namely a Poincaré inequality and a Bonnet--Myers theorem.

It is well known that when $K>0$, the dimensionless bound $\BE(K,\infty)$ implies that the smallest non-zero eigenvalue of the generator is at least $K$. As a simple application of the dimensional variant we show that this bound can be improved.

\begin{proposition}[Poincaré inequality]
	Let $K>0$ and $N>1$. If $(P_t)$ satisfies $\BE(K,N)$ and $\lambda_1$ is the smallest non-zero eigenvalue of $\mathcal{L}$, then
	\begin{equation*}
	\lambda_1\geq \frac{KN}{N-1}.
	\end{equation*}
\end{proposition}
\begin{proof}
	By $\mathrm{BE}(K,N)$ we have
	\begin{align*}
	\norm{\mathcal{L} a}_2^2=\tau(\Gamma_2(a))\geq K\tau(\Gamma(a))+\frac 1 N \tau(\abs{\mathcal{L}a}^2)=K\langle \mathcal{L} a,a\rangle_2+\frac 1 N \norm{\mathcal{L}a}_2^2.
	\end{align*}
	In particular, if $\mathcal{L}a=\lambda_1 a$ and $\norm{a}_2=1$, then
	\begin{align*}
	\lambda_1^2\geq K\lambda_1+\frac 1 N \lambda_1^2,
	\end{align*}
	from which the desired inequality follows.
\end{proof}

To state the Bonnet--Myers theorem, we recall the definition of the metric $d_\Gamma$ on the space of density matrices that is variously known as quantum $L^1$-Wasserstein distance, Connes distance or spectral distance. It is given by
\begin{equation*}
d_\Gamma(\rho_0,\rho_1)=\sup\{\tau(a(\rho_1-\rho_0))\mid a=a^\ast\in\M,\,\Gamma(a)\leq \un\}
\end{equation*}
for $\rho_0,\rho_1\in\states$.

\begin{proposition}\label{prop:Bonnet-Myers_BE}
Let $K,N\in (0,\infty)$. If a symmetric quantum Markov semigroup $(P_t)$ is ergodic and satisfies Bakry--Émery curvature-dimension condition $\BE(K,N)$, then
	\begin{equation*}
	d_\Gamma(\rho,\un)\le\frac\pi 2\sqrt{\frac{N}{K}}
	\end{equation*}
	for all $\rho\in\mathcal{S}(\M)$.
	
In particular,
\begin{equation*}
\sup_{\rho_0,\rho_1\in\states}d_\Gamma(\rho_0,\rho_1)\leq \pi\sqrt{\frac K N}.
\end{equation*}
\end{proposition}

\begin{proof}
The proof follows the same line as that of \cite[Theorem 2.4]{LMP18}. The condition $\mathrm{BE}(K,N)$ implies 
\begin{align*}
\frac{1-e^{-2Kt}}{KN}(\L P_t a)^2\leq e^{-2Kt}P_t\Gamma(a),
\end{align*}
for any $a=a^\ast \in \cM$. If $\Gamma(a)\leq \un$, we have
\begin{align*}
\norm{\L P_t a}_\infty\leq \sqrt{KN}\sqrt{\frac{1}{e^{2Kt}-1}}.
\end{align*}
Thus for any $\rho\in\mathcal{S}(\M)$,
\begin{align*}
\abs{\tau((P_t a-a)\rho)}\leq \int_0^t \left\lvert\frac{d}{ds}\tau((P_s a)\rho)\right\rvert\,ds\leq\sqrt{KN}\int_0^\infty \frac 1{\sqrt{e^{2Ks}-1}}\,ds=\frac{\pi}{2}\sqrt{\frac{N}{K}}.
\end{align*}
Therefore
\begin{align*}
\tau(a(\rho-\un))=\tau((a-P_t a)\rho)+\tau(P_t a(\rho-\un))
\leq \frac{\pi}{2}\sqrt{\frac{N}{K}}+\tau(a(P_t\rho-\un)).
\end{align*}
Since $(P_t)$ is assumed to be ergodic, we have $P_t\rho\to \un$ as $t\to\infty$, and we end up with
\begin{equation*}
d_\Gamma(\rho,\un)= \sup_{\Gamma(a)\leq 1}\tau(a(\rho-\un))\leq \frac{\pi}{2}\sqrt{\frac{N}{K}}.\qedhere
\end{equation*}

\end{proof}

\subsection{\texorpdfstring{Complete BE($K,N$)}{Complete BE(K,N)}}

In many applications it is desirable to have estimates that are tensor-stable in the sense that they hold not only for $(P_t)$, but also for $(P_t\otimes \mathrm{id}_{M_n(\IC)})$ with a constant independent of $n\in\IN$. Even in the case $K=\infty$, it seems to be unknown if this is true for the Bakry--Émery estimate. For that reason we introduce the complete Bakry--Émery estimate $\CBE(K,N)$, which has this tensor stability by definition. We will show that this stronger estimate also holds for quantum Markov semigroup satisfying the $K$-intertwining condition, and moreover, this estimate behaves as expected under arbitrary tensor products.

\begin{definition}
Let $K\in\IR$ and $N> 0$. We say that the quantum Markov semigroup $(P_t)$ satisfies $\CBE(K,N)$ if
	\begin{align*}
	[\Gamma(P_t x_j,P_t x_k)]_{j,k}\leq e^{-2Kt}[P_t \Gamma(x_j,x_k)]_{j,k}-\frac{1-e^{-2Kt}}{KN} [(\L P_t x_j)^\ast (\L P_t x_k)]_{j,k},
	\end{align*}
	for all $x_1,\dots,x_n\in \cM$ and $t> 0$.
\end{definition}

Just as in Proposition \ref{prop:char_BE} one can show that $\CBE(K,N)$ is equivalent to
\begin{equation*}
[\Gamma_2(x_j,x_k)]_{j,k}\geq K[\Gamma(x_j,x_k)]_{j,k}+\frac 1 N [(\L x_j)^\ast (\L x_k)]_{j,k}
\end{equation*}
for all $x_1,\dots,x_n\in \M$ and $t\geq 0$.

For $N=\infty$, this criterion was introduced in \cite{JZ15a} for group von Neumann algebras under the name \emph{algebraic $\Gamma_2$-condition}.

To show that $\CBE(K,N)$ for $(P_t)$ is equivalent to $\BE(K,N)$ for $(P_t\otimes\id_{M_n(\IC)})$ with constants independent of $n$, we need the following elementary lemma.

\begin{lemma}\label{lem:positive}
	Let $\mathcal{A},\mathcal{B}$ be two C*-algebras. If $x=[x_{jk}]\in M_n(\mathcal{A})$, $y=[y_{jk}]\in M_n(\mathcal{B})$ are positive, then
	\begin{equation*}
	\sum_{j,k}x_{jk}\otimes y_{jk}\geq 0.
	\end{equation*}
\end{lemma}
\begin{proof}
	By assumption there are $a=[a_{jk}]\in M_n(\mathcal{A})$, $b=[b_{jk}]\in M_n(\mathcal{B})$ such that
	\begin{align*}
	x_{jk}&=\sum_{l}a_{lj}^\ast a_{lk},\\
	y_{jk}&=\sum_{m}b_{mj}^\ast b_{mk}.
	\end{align*}
	Thus
	\begin{align*}
	\sum_{j,k}x_{jk}\otimes y_{jk}&=\sum_{j,k,l,m}a_{lj}^\ast a_{lk}\otimes b_{mj}^\ast b_{mk}\\
	&=\sum_{l,m}\left(\sum_j a_{lj}^\ast \otimes b_{mj}^\ast\right)\left(\sum_k a_{lk}\otimes b_{mk}\right)\\
	&=\sum_{l,m}\left\lvert \sum_j a_{lj}\otimes b_{mj}\right\rvert^2\\
	&\geq 0.\qedhere
	\end{align*}
\end{proof}

\begin{proposition}\label{prop:char_CBE}
	Let $(P_t)$ be a quantum Markov semigroup on $\M$. For $K\in \IR$ and $N\in(0,\infty]$, the following assertions are equivalent:
	\begin{enumerate}[(a)]
		\item $(P_t)$ satisfies $\CBE(K,N)$.
%		\item $(P_t\otimes\mathrm{id}_\N)$ satisfies $\BE(K,N)$ for all tracial $(\N,\tau)$.
		\item $(P_t\otimes\id_{M_n(\IC)})$ satisfies $\BE(K,N)$ for all $n\in\IN$.
	\end{enumerate}
\end{proposition}
\begin{proof}
	(a)$\implies$(b): Write $\Gamma,\Gamma_2$ for the (iterated) carré du champ associated with $(P_t)$ and $\Gamma^\otimes,\Gamma_2^\otimes$ for the same forms associated with $(P_t\otimes\id_{M_n(\IC)})$.
	
	A direct computation shows
	\begin{align*}
	\Gamma_2^\otimes\left(\sum_j x_j\otimes y_j\right)&=\sum_{j,k}\Gamma_2(x_j,x_k)\otimes y_j^\ast y_k,\\
	\Gamma^\otimes \left(\sum_j x_j\otimes y_j\right)&=\sum_{j,k}\Gamma(x_j,x_k)\otimes y_j^\ast y_k,\\
	\left\lvert(\L\otimes\id_\N)\left(\sum_j x_j\otimes y_j\right)\right\rvert^2&=\sum_{j,k}(\L x_j)^\ast (\L x_k)\otimes y_j^\ast y_k.
	\end{align*}
	Hence
	\begin{align*}
	&\quad\;\Gamma_2^\otimes\left(\sum_j x_j\otimes y_j\right)-K\Gamma^\otimes(\sum_j x_j\otimes y_j)-\frac 1 N \left\lvert(\L\otimes\id_\N)\left(\sum_j x_j\otimes y_j\right)\right\rvert^2\\
	&=\sum_{j,k}(\Gamma_2(x_j,x_k)-K\Gamma(x_j,x_k)-\frac 1 N(\L x_j)^\ast (\L x_k))\otimes y_j^\ast y_k,
	\end{align*}
	and the result follows from Lemma \ref{lem:positive} and (a).
	
	(b)$\implies$(a): Let $x=\sum_j x_j\otimes \ket{1}\bra{j}$. The computations from (a)$\implies$(b) show
	\begin{equation*}
	\Gamma_2^\otimes(x)=\sum_{j,k}\Gamma_2(x_j,x_k)\otimes \ket{j}\bra{k}
	\end{equation*}
	and similar formulas for $\Gamma^\otimes$ and $\L\otimes\id_{M_n(\IC)}$. Using the $\ast$-isomorphism $\M\otimes M_n(\IC)\to M_n(\M),\,\sum_{j,k}x_{jk}\otimes\ket{j}\bra{k}\mapsto [x_{jk}]_{j,k}$, assertion (a) follows.
\end{proof}

In the following two results we will give two classes of examples for which the condition $\CBE$ is satisfied.

\begin{proposition}\label{prop:intertwining implies CBE}
Suppose that the generator $\cL$ of the quantum Markov semigroup $(P_t)$ admits the Lindblad form \eqref{eq:lindblad} with $d$ partial derivatives $\partial_1,\dots,\partial_d$. If $(P_t)$ satisfies the $K$-intertwining condition for $K\in\mathbb{R}$, then $(P_t)$ satisfies $\CBE(K,d)$.
\end{proposition}
\begin{proof}
A direct computation shows that $\L\otimes\mathrm{id}_{M_n(\IC)}$ admits a Lindblad form with derivations $\partial_1\otimes \mathrm{id}_{M_n(\IC)},\dots,\partial_d\otimes\mathrm{id}_{M_n(\IC)}$. Now the claim is a direct consequence of Propositions \ref{prop:BE_intertwining} and \ref{prop:char_CBE}.
\end{proof}

\begin{proposition}
	If $\M$ is commutative and $(P_t)$ satisfies $\BE(K,N)$, then it also satisfies $\CBE(K,N)$.
\end{proposition}
\begin{proof}
	By assumption, $\M\cong C(X)$ for a compact space $X$. We have to show
	\begin{align*}
	[\Gamma_2(f_j,f_k)(x)]_{j,k}\geq K[\Gamma(f_j,f_k)(x)]_{j,k}+\frac 1 N [\overline{(\L f_j)(x)} (\L f_k)(x)]_{j,k}
	\end{align*}
for $x\in X$,
	which follows from
	\begin{align*}
	\sum_{j,k}\overline{\alpha_j}\alpha_k \Gamma_2(f_j,f_k)(x)&=\Gamma_2\left(\sum_j \alpha_j f_j\right)(x)\\
	&\geq K\Gamma\left(\sum_j \alpha_j f_j\right)(x)+\frac 1 N \left\lvert \L\left(\sum_j \alpha_j f_j\right)(x)\right\rvert^2\\
	&=K\sum_{j,k}\overline{\alpha_j}\alpha_k \Gamma(f_j,f_k)(x)+\frac 1 N \sum_{j,k}\overline{\alpha_j}\alpha_k \overline{(\L f_j)(x)}(\L f_k)(x)
	\end{align*}
	for any $\alpha_j\in \IC$.
\end{proof}

Before we state the tensorization property of $\CBE$, we need another elementary inequality.
\begin{lemma}\label{lem:Young}
	Let $\mathcal{A}$ be a C*-algebra. If $a,b\in \mathcal{A}$ and $\lambda>0$, then
	\begin{equation*}
	\abs{a+b}^2\leq (1+\lambda)\abs{a}^2+(1+\lambda^{-1})\abs{b}^2.
	\end{equation*}
\end{lemma}
\begin{proof}
	In fact, 
	\begin{align*}
	(1+\lambda)\abs{a}^2+(1+\lambda^{-1})\abs{b}^2
	&=|a+b|^2+\lambda \abs{a}^2+\lambda^{-1}\abs{b}^2-a^\ast b-b^\ast a\\
	&=|a+b|^2+\abs{\lambda^{1/2}a-\lambda^{-1/2}b}^2\\
	&\ge |a+b|^2.\qedhere
	\end{align*}
	%We have
	%\begin{equation*}
	%0\leq\abs{\lambda^{1/2}a-\lambda^{-1/2}b}^2=\lambda %\abs{a}^2+\lambda^{-1}\abs{b}^2-a^\ast b-b^\ast a.
	%\end{equation*}
	%Thus $a^\ast b+b^\ast a\leq \lambda\abs{a}^2+\lambda^{-1}\abs{b}^2$, and 
	%\begin{equation*}
	%\abs{a+b}^2
	%=|a|^2+|b|^2+a^\ast b+b^\ast a
	%\leq (1+\lambda)\abs{a}^2+(1+\lambda^{-1})\abs{a}^2.
	%\end{equation*}
	%the desired inequality follows by expanding the absolute value.
\end{proof}

\begin{proposition}
Let $\M$, $\N$ be finite-dimensional von Neumann algebras and let $(P_t)$, $(Q_t)$ be tracially symmetric quantum Markov semigroups on $\M$ and $\N$, respectively. If $(P_t)$ satisfies $\CBE(K,N)$ and $(Q_t)$ satisfies $\CBE(K',N')$, then $(P_t\otimes Q_t)$ satisfies $\CBE(\min\{K,K'\},N+N')$.
\end{proposition}
\begin{proof}
We use superscripts for the (iterated) carré du champ to indicate the associated quantum Markov semigroup. Let $\kappa=\min\{K,K'\}$. We have
\begin{align*}
\Gamma_2^{P\otimes Q}-\kappa\Gamma^{P\otimes Q}
=(\Gamma_2^{P\otimes \mathrm{id}_\N}-\kappa\Gamma^{P\otimes \mathrm{id}_\N})+(\Gamma_2^{\mathrm{id}_\M\otimes Q}-\kappa\Gamma^{\mathrm{id}_\M\otimes Q})+2\Gamma^P\otimes\Gamma^Q,
\end{align*}
where 
\begin{equation*}
(\Gamma^P\otimes \Gamma^Q)\left(\sum_j x_j\otimes y_j\right):=\sum_{j,k}\Gamma^P(x_j,x_k)\otimes \Gamma^Q(y_j,y_k).
\end{equation*}
By $\CBE(\kappa,N)$ for $(P_t)$ and $\CBE(\kappa,N')$ for $(Q_t)$ we have
\begin{align*}
(\Gamma_2^{P\otimes \mathrm{id}_\N}-\kappa\Gamma^{P\otimes \mathrm{id}_\N})\left(\sum_j x_j\otimes y_j\right)\geq \frac 1 N\left\lvert\sum_j \L_Px_j\otimes y_j\right\rvert^2,\\
(\Gamma_2^{\mathrm{id}_\M\otimes Q}-\kappa\Gamma^{\mathrm{id}_\M\otimes Q})\left(\sum_j x_j\otimes y_j\right)\geq \frac 1 {N'}\left\lvert\sum_j x_j\otimes \L_Q y_j\right\rvert^2.
\end{align*}
Moreover, 
\begin{equation*}
(\Gamma^P\otimes \Gamma^Q)\left(\sum_j x_j\otimes y_j\right)\geq 0
\end{equation*}
by Lemma \ref{lem:positive}.

Finally,
\begin{equation*}
\frac 1 N\left\lvert\sum_j \L_Px_j\otimes y_j\right\rvert^2+\frac 1 {N'}\left\lvert\sum_j x_j\otimes \L_Q y_j\right\rvert^2\geq \frac 1{N+N'}\left\lvert\sum_j \L_P x_j\otimes y_j+x_j\otimes \L_Q y_j\right\rvert^2
\end{equation*}
by Lemma \ref{lem:Young}, which shows $\BE(\kappa,N+N')$ for $(P_t\otimes Q_t)$. To prove $\CBE(\kappa,N+N')$, we can simply apply the same argument to $(P_t\otimes\mathrm{id}_{M_n(\IC)})$ and $(Q_t\otimes\mathrm{id}_{M_n(\IC)})$ for arbitrary $n\in \IN$.
\end{proof}

\section{\texorpdfstring{The gradient estimate $\GE(K,N)$}{The gradient estimate GE(K,N)}}
\label{sec:CD GE}

\subsection{\texorpdfstring{Gradient estimate $\GE(K,N)$ and a sufficient condition}{Gradient estimate GE(K,N) and a sufficient condition}}
In \cite{CM14,CM17,Wir18,CM20}, a noncommutative analog of the $2$-Wasserstein metric was constructed on the set of quantum states. Among other things, it gives rise to a notion of (entropic) lower Ricci curvature bound via geodesic semi-convexity of the entropy. This allows to prove a number of functional inequalities under strictly positive lower Ricci curvature bound, including the modified log-Sobolev inequality that (seemingly) cannot be produced under the Bakry--Émery curvature-dimension condition BE($K,\infty$).

This entropic lower Ricci curvature bound is captured in the following gradient estimate
\begin{equation*}\label{ineq:GE(K,infty)}
\norm{\partial P_t a}_\rho^2
\le e^{-2Kt}\norm{\partial a}_{P_t\rho}^2,\tag{GE$(K,\infty)$}
\end{equation*}
or equivalently
\begin{equation}\label{ineq:CD(K,infty)}
\Re \langle \partial \cL a,\hat{\rho}\partial a\rangle
+\frac{1}{2}\left\langle \frac{d}{dt}\big|_{t=0}\widehat{P_t \rho}\partial a,\partial a\right\rangle 
\ge K \norm{\partial a}_{\rho}^2,
\end{equation}
where the notations $\hat{\rho}$ and $\|\cdot\|_{\rho}$ correspond to the logarithmic mean $\Lambda_{\text{log}}$. Recall Section \ref{sec:QMS} for more details. The fact that logarithmic mean comes into play lies in the use of chain rule
\begin{equation*}
\hat{\rho}\partial_j \log \rho=\partial_j\rho,~~1\le j\le d.
\end{equation*}

In fact, for the gradient estimate \eqref{ineq:GE(K,infty)} and its equivalent form \eqref{ineq:CD(K,infty)} one can work with any operator mean.
% any of which associates with a noncommutative transport metric, see Wirth?. 
This not only makes the theory more flexible, but also includes the Bakry--Émery curvature-dimension condition BE($K,\infty$) as a special case. Indeed, one recovers BE($K,\infty$) by replacing the logarithmic mean in \eqref{ineq:CD(K,infty)} with the left/right trivial mean. In the next section we discuss the connection of GE($K,N$) and $(K,N)$-convexity of the (relative) entropy.
%So we postpone to next section the approach via geodesic $(K,N)$-convexity of entropy, which corresponds to the logarithmic mean. 

The study of \eqref{ineq:GE(K,infty)} for arbitrary operator means was started in \cite{Wir18,WZ20}. Here we continue to work within this framework and focus on the ``finite-dimensional'' version of \eqref{ineq:GE(K,infty)} or \eqref{ineq:CD(K,infty)}, which we call \emph{gradient estimate} $\GE(K,N)$.

\begin{definition}
	Let $\Lambda$ be an operator mean and $(P_t)$ be a symmetric quantum Markov semigroup whose generator takes the Lindblad form \eqref{eq:lindblad}. We say that $(P_t)$ satisfies the gradient estimate $\GE(K,N)$ for $K\in \mathbb{R}, N\in (0,\infty]$ if
	\begin{equation*}\label{ineq:GE(K,N)}
	\norm{\partial P_t a}_\rho^2
	\le e^{-2Kt}\norm{\partial a}_{P_t\rho}^2-\frac{1-e^{-2Kt}}{KN} \abs{\E( a,P_t\rho)}^2,\tag{GE$(K,N)$}
	\end{equation*}
	for any $t\ge 0$, $a\in \cM$ and $\rho\in \mathcal{S}_+(\cM)$.
\end{definition}

\begin{remark}
	It is obvious that when $N=\infty$, \eqref{ineq:GE(K,N)} becomes the gradient estimate $\GE(K,\infty)$. From the definition it is not immediately clear that if $(P_t)$ satisfies the gradient estimate $\GE(K,N)$, then it also satisfies the gradient estimate $\GE(K',N')$ whenever $K'\le K$ and $N'\ge N$. But this can be seen from the following equivalent formulation in the flavor of the $\Gamma_2$-condition.
\end{remark}

\begin{remark}
If $\rho\in \states$ is not invertible, one can apply \ref{ineq:GE(K,N)} to $\rho^\epsilon=\frac{\rho+\epsilon\un}{1+\epsilon}$ and let $\epsilon\searrow 0$ to see that \ref{ineq:GE(K,N)} still remains true.
\end{remark}

\begin{proposition}\label{prop:equiv_GE_CD}
	For any operator mean $\Lambda$ and any symmetric quantum Markov semigroup $(P_t)$, the gradient estimate \ref{ineq:GE(K,N)} holds if and only if
	\begin{equation}\label{ineq:GE(K,N) equiv}
	\Re \langle \partial \cL a,\hat{\rho}\partial a\rangle
	-\frac{1}{2}\langle dG(\rho)(\L\rho)\partial a,\partial a\rangle 
	\ge K \norm{\partial a}_{\rho}^2+\frac{1}{N}\abs{\E( a,\rho)}^2
	\end{equation}
	for any $\rho\in\mathcal{S}_+(\M)$ and any $a\in\M$. Here $dG(\rho)$ denotes the Fréchet derivative of $G(\rho):=\hat{\rho}=\Lambda(L(\rho),R(\rho))$. 
\end{proposition}

\begin{proof}
	Assume that $(P_t)$ satisfies \ref{ineq:GE(K,N)}. Set 
	\begin{equation*}
	\phi(t):=e^{-2Kt}\norm{\partial a}_{P_t\rho}^2-\norm{\partial P_t a}_\rho^2-\frac{1-e^{-2Kt}}{KN}\abs{\E(a,P_t \rho)}^2.
	\end{equation*}
	Then $\phi(t)\ge 0$ and $\phi(0)=0$. Therefore $\phi'(0)\ge 0$, that is, 
	\begin{equation*}
	\left\langle \frac{d}{dt}\big|_{t=0}\widehat{P_t \rho}\partial a,\partial a\right\rangle
	-2K\norm{\partial a}_{\rho}^2
	+\langle \hat{\rho}\partial \cL a,\partial a\rangle
	+\langle \hat{\rho}\partial a,\partial \cL a\rangle
	-\frac{2}{N}\abs{\E( a,\rho)}^2\ge 0.
	\end{equation*}
	This is nothing but \eqref{ineq:GE(K,N) equiv}, since $dG(\rho)(\L\rho)=-\frac{d}{dt}\big|_{t=0}\widehat{P_t \rho}$.
	
	Now suppose that $(P_t)$ satisfies \eqref{ineq:GE(K,N) equiv}. Fix $t> 0$ and put 
	\begin{equation*}
	\varphi(s):=e^{-2Ks}\norm{\partial P_{t-s} a}_{P_{s}\rho}^2,~~0\le s\le t.
	\end{equation*}
	Then applying \eqref{ineq:GE(K,N) equiv} to $(\rho,a)=(P_s\rho,P_{t-s}a)$, we get
	\begin{align*}
	\varphi'(s)=&
	e^{-2Ks}\left(\langle \widehat{P_s\rho}\partial \cL P_{t-s}a,\partial P_{t-s}a\rangle
	+\langle \widehat{P_s\rho}\partial P_{t-s} a,\partial \cL P_{t-s} a\rangle\right.\\
	&\left.-\left\langle dG(P_s\rho)(\L\rho)\partial P_{t-s}a,\partial P_{t-s} a\right\rangle-2K\norm{\partial P_{t-s} a}_{P_{s}\rho}^2\right)\\
	\ge &\frac{2}{N}e^{-2Ks}\abs{\E(P_{t-s} a,P_s\rho)}^2\\
	=&\frac{2}{N}e^{-2Ks}\abs{\E(a,P_{t}\rho)}^2.
	\end{align*}
	This, together with the fundamental theorem of calculus, yields
	\begin{align*}
	e^{-2Kt}\norm{\partial a}_{P_t\rho}^2-	\norm{\partial P_t a}_\rho^2
	=\varphi(t)-\varphi(0)
	=\int_{0}^{t}\varphi'(s)ds
	\ge \frac{1-e^{-2Kt}}{KN}\abs{\E(a,P_t\rho)}^2.
	\end{align*}
	Therefore $(P_t)$ satisfies \ref{ineq:GE(K,N)}.
\end{proof}

\begin{remark}\label{rk:implicit constant_GE}
	Similar to Remark \ref{rk:implicit constant_BE}, the function
	$$t\mapsto \frac{1-e^{-2Kt}}{KN},$$ 
	in \ref{ineq:GE(K,N)} can be replaced by any $f$ such that $f(0)=0$ and $f'(0)=2/N$.
\end{remark}

\begin{remark}
	In the case $N=\infty$, the gradient estimate $\GE(K,\infty)$ for the left trivial mean is equivalent to the exponential form of $\BE(K,\infty)$. For $N<\infty$ this seems to be no longer the case, but one still has one implication: the Bakry--Émery curvature-dimension condition BE($K,N$) is stronger than $\GE(K,N)$ for the left trivial mean. This is a consequence of Cauchy-Schwarz inequality for the state $\tau(\rho\,\cdot\,)$:
	\begin{equation*}
	|\E(a,\rho)|^2=|\langle \cL a,\rho\rangle|^2\le \langle |\cL a|^2,\rho\rangle.
	\end{equation*}
\end{remark}

Similar to BE($K,N$), the intertwining condition is also sufficient to prove $\GE(K,N)$ with the same dimension (upper bound).

\begin{theorem}\label{thm:CD under intertwining}
	Let $(P_t)$ be a symmetric quantum Markov semigroup over $\cM$ with the Lindblad form \eqref{eq:lindblad}. Suppose that $(P_t)$ satisfies $K$-intertwining condition for some $K\in \mathbb{R}$. Then for any operator mean $\Lambda$ the quantum Markov semigroup $(P_t)$ satisfies $\GE(K,d)$. %with $K:=\min_{1\le j\le d} K_j$.
\end{theorem}

\begin{proof}
	For $a\in \cM$, recall that
	\begin{align*}
	P_t(a^\ast a)-(P_t a)^\ast P_t a
	=\int_{0}^{t}\frac{d}{ds} P_{s}\left((P_{t-s} a)^\ast P_{t-s} a\right) ds
	=2\int_{0}^{t} P_{s}\Gamma(P_{t-s} a) ds.
	\end{align*}
	Under the $K$-intertwining condition, we have (either by Kadison--Schwarz or BE($K,\infty$))
	$$P_{s}\Gamma(P_{t-s} a)\ge e^{2Ks}\Gamma(P_t a).$$
	So
	\begin{equation}\label{ineq:from BE(K,infty)}
	P_t(a^\ast a)-(P_t a)^\ast P_t a
	\ge 2\int_{0}^{t}e^{2Ks}ds\Gamma(P_t a)
	=\frac{e^{2Kt}-1}{K}\Gamma(P_t a).
	\end{equation}
	By \eqref{eq:gamma interms of partial} and Lemma \ref{lem:simple ineq}, we get for any $(x_j)_{1\le j\le d}\subset \cM$
	\begin{equation}\label{ineq:sum of gradients}
	\sum_{j=1}^{d}\Gamma(P_t x_j)
	=\sum_{j,k=1}^{d}|\partial_k P_t x_j|^2
	=\sum_{j,k=1}^{d}|\partial_k^\dagger P_t x_j|^2
	\ge\sum_{j=1}^{d}|\partial_j^\dagger P_t x_j|^2
	\ge\frac{1}{d}\left|\sum_{j=1}^{d}\partial_j^\dagger P_t x_j\right|^2.
	\end{equation}
	Let $\hat{\cM}=\oplus_{j=1}^{d}\cM$ be equipped with the inner product 
	$$\langle (x_j),(y_j)\rangle:=\sum_{j=1}^{d}\langle x_j,y_j\rangle,$$
	and $\hat P_t $ be the operator acting on $\hat{\cM}$ such that $\hat P_t(x_1,\dots,x_d)=(P_t x_1,\dots, P_t x_d)$. Fix $\rho\in\mathcal{S}_+(\cM)$. For simplicity, let us identify $\rho$ with the element $(\rho,\dots,\rho)$ in $\hat{\cM}$. Then for $x=(x_1,\dots,x_d)\in \hat{\cM}$, we have by \eqref{ineq:from BE(K,infty)} and \eqref{ineq:sum of gradients} that
	\begin{align*}
	\langle \hat P_t(x^\ast x),\rho\rangle-\langle (\hat P_t x)^\ast \hat P_t x,\rho\rangle
	=&\sum_{j=1}^{d}\langle P_t(x_j^\ast x_j)-(P_t x_j)^\ast P_t x_j,\rho\rangle \\
	\ge& \frac{e^{2Kt}-1}{dK}\sum_{j=1}^{d}\langle \Gamma (P_t x_j),\rho\rangle\\
	\ge &\frac{e^{2Kt}-1}{dK}\left\langle \left|\sum_{j=1}^{d} \partial_j^\dagger P_t x_j\right|^2,\rho\right\rangle.
	\end{align*}
	From $K$-intertwining condition and Cauchy-Schwarz inequality for the state $\tau(\rho\cdot)$ on $\cM$, this is bounded from below by
	%\begin{align*}
	%\langle \hat P_t(x^\ast x),\rho\rangle-\langle (\hat P_t x)^\ast \hat P_t x,\rho\rangle
	%\ge & \frac{1-e^{-2Kt}}{2dK}\left\langle \left|\sum_{j=1}^{d} P_t\partial_j^\dagger %x_j\right|^2,\rho\right\rangle\\
	%\ge & \frac{1-e^{-2Kt}}{2dK}\left|\sum_{j=1}^{d}\left\langle P_t\partial_j^\dagger x_j,\rho\right\rangle\right|^2\\
	%= & \frac{1-e^{-2Kt}}{2dK}\left|\left\langle  x,\partial P_t  \rho\right\rangle\right|^2.
	%\end{align*}
	\begin{equation*}
	\frac{1-e^{-2Kt}}{dK}\left\langle \left|\sum_{j=1}^{d} P_t\partial_j^\dagger x_j\right|^2,\rho\right\rangle
	\ge\frac{1-e^{-2Kt}}{dK}\left|\sum_{j=1}^{d}\left\langle P_t\partial_j^\dagger x_j,\rho\right\rangle\right|^2\\
	= \frac{1-e^{-2Kt}}{dK}\left|\left\langle  x,\partial P_t  \rho\right\rangle\right|^2.
	\end{equation*}
	So we have proved that for any $x\in \hat{\cM}$:
	\begin{align*}
	\langle x(\hat P_t \rho),x\rangle
	\ge \langle \hat P_t x,(\hat P_t x)\rho\rangle+\frac{1-e^{-2Kt}}{dK}\langle x,\ket{\partial P_t \rho}\bra{\partial P_t\rho}(x)\rangle,
	\end{align*}
	or equivalently
	\begin{equation*}
	R(P_t \rho)\geq \hat P_t R(\rho)\hat P_t +\frac{1-e^{-2Kt}}{dK}\ket{\partial P_t \rho}\bra{\partial P_t\rho}.
	\end{equation*}
	Replacing $x$ by $x^\ast$, we obtain
	\begin{equation*}
	L(P_t \rho)\geq \hat P_t L(\rho)\hat P_t+\frac{1-e^{-2Kt}}{dK}\ket{\partial P_t \rho}\bra{\partial P_t\rho}.
	\end{equation*}
	Note that the second summand is the same in both cases.
	
	Now since $\Lambda$ is an operator mean, we have
	\begin{align*}
	\Lambda(L(P_t\rho),R(P_t\rho)) 
	\geq &\Lambda (\hat P_t L(\rho)\hat P_t,\hat P_t R(\rho)\hat P_t)+\frac{1-e^{-2Kt}}{dK}\Lambda\left(\ket{\partial P_t \rho}\bra{\partial P_t\rho},\ket{\partial P_t \rho}\bra{\partial P_t\rho}\right)\\
	\ge &\hat P_t \Lambda(L(\rho),R(\rho))\hat P_t+\frac{1-e^{-2Kt}}{dK}\ket{\partial P_t \rho}\bra{\partial P_t\rho},
	\end{align*}
	where in the first inequality we used the monotonicity, concavity (Lemma \ref{lem:mean} (b)) and positive homogeneity (Lemma \ref{lem:mean} (a)) of $\Lambda$, and in the second inequality we used the transformer inequality and Lemma \ref{lem:mean}(d).
	This, together with the $K$-intertwining condition, yields
	\begin{align*}
	\norm{\partial P_t a}_\rho^2&=\langle \Lambda(L(\rho),R(\rho))\partial P_t a,\partial P_t a\rangle\\
	&= e^{-2Kt}\langle \hat P_t \Lambda(L(\rho),R(\rho))\hat P_t \partial a,\partial a\rangle\\
	&\leq e^{-2Kt}\langle \Lambda(L(P_t\rho),R(P_t\rho))\partial a,\partial a\rangle - \frac{e^{-2Kt}-e^{-4Kt}}{dK}\abs{\langle \partial P_t \rho,\partial a\rangle}^2\\
	&=e^{-2Kt}\norm{\partial a}_{P_t\rho}^2-\frac{e^{-2Kt}-e^{-4Kt}}{dK} \abs{\E(a,P_t\rho)}^2.
	\end{align*}
	This completes the proof, by Remark \ref{rk:implicit constant_GE}.
\end{proof}

\subsection{Bonnet--Myers theorem}\label{subsec:Bonnet-Myers}

As a first application of the dimensional gradient estimate $\GE(K,N)$, we present here a Bonnet--Myers theorem for the noncommutative analog of the Wasserstein distance introduced in \cite{CM17,CM20}. The proof is quite similar (or, in fact, similar to the dual) to the proof of Proposition \ref{prop:Bonnet-Myers_BE}.

Let us first recall the definition of the metric. The space $\mathcal{S}_+(\M)$ of invertible density matrices is a smooth manifold and the tangent space at $\rho\in \mathcal{S}_+(\M)$ can be canonically identified with the traceless self-adjoint elements of $\mathcal{M}$. Assume that $(P_t)$ is a tracially symmetric quantum Markov semigroup with generator $\L$ with Lindblad form (\ref{eq:lindblad}).

Fix an operator mean $\Lambda$. For $\rho\in\mathcal S_+(\M)$ we define
\begin{equation}\label{eq:Riemannian operator}
\K_\rho^\Lambda\colon \M\to \M,\,x\mapsto\partial^\dagger\hat{\rho}\partial x=\sum_{j=1}^d\partial_j^\dagger(\Lambda(L(\rho),R(\rho))\partial_j(x)).
\end{equation}
The Riemannian metric $g^\Lambda$ on $\mathcal{S}_+(\M)$ is defined by
\begin{equation*}
g^\Lambda_\rho(\dot \rho_1,\dot \rho_2)=\langle \dot\rho_1,(\K^\Lambda_\rho)^{-1}\dot\rho_2\rangle.
\end{equation*}
The associated distance function on $\mathcal{S}_+(\M)\times \mathcal{S}_+(\M)$ is denoted by $\W$. By \cite[Proposition 9.2]{CM20}, $\W$ can be extended to $\mathcal{S}(\M)\times \mathcal{S}(\M)$ since
\begin{equation*}
\Lambda(a\un,b\un)\geq\Lambda(\min\{a,b\}\un,\min\{a,b\}\un)=\min\{a,b\}\un
\end{equation*}
for all $a,b>0$.

\begin{proposition}
Fix an operator mean $\Lambda$. Let $K,N\in (0,\infty)$. If $(P_t)$ is ergodic and satisfies gradient estimate $\GE(K,N)$, then
	\begin{equation*}
	\W(\rho,\un)\le\frac\pi 2\sqrt{\frac{N}{K}}
	\end{equation*}
	for all $\rho\in\mathcal{S}_+(\M)$.
	
In particular,
\begin{equation*}
\sup_{\rho_0,\rho_1\in\mathcal{S}_+(\M)}\W(\rho_0,\rho_1)\leq \pi\sqrt{\frac K N}.
\end{equation*}
\end{proposition}
\begin{proof}
Since $(P_t)$ is ergodic, we have $P_t \rho\to \un$ as $t\to\infty$. Let $\rho_t=P_t \rho$ for $t\geq 0$. The gradient estimate $\GE(K,N)$ implies
\begin{equation*}
\abs{\langle a,\dot \rho_t\rangle}=\abs{\langle a,\L\rho_t\rangle}\leq \sqrt{KN}\sqrt{\frac{e^{-2Kt}}{1-e^{-2Kt}}}\norm{\partial a}_{\rho_t}=\frac{\sqrt{KN}}{\sqrt{e^{2Kt}-1}}\norm{\partial a}_{\rho_t}
\end{equation*}
for all $a\in \M$. Choosing $a=(\K^\Lambda_{\rho_t})^{-1}\dot\rho_t$, we get
\begin{equation*}
g^{\Lambda}_{\rho_t}(\dot\rho_t,\dot \rho_t)
\leq \frac{\sqrt{KN}}{\sqrt{e^{2Kt}-1}}\sqrt{\langle\K^\Lambda_{\rho_t} a,a\rangle}
= \frac{\sqrt{KN}}{\sqrt{e^{2Kt}-1}}\sqrt{g^{\Lambda}_{\rho_t}(\dot\rho_t,\dot \rho_t)}.
\end{equation*}
Hence
\begin{equation*}
\sqrt{g^{\Lambda}_{\rho_t}(\dot\rho_t,\dot \rho_t)}\leq \frac{\sqrt{KN}}{\sqrt{e^{2Kt}-1}},
\end{equation*}
and we conclude
\begin{equation*}
\W(\rho,\un)\leq \int_0^\infty\sqrt{g^{\Lambda}_{\rho_t}(\dot\rho_t,\dot \rho_t)}\,dt\leq \int_0^\infty \frac{\sqrt{KN}}{\sqrt{e^{2Kt}-1}}\,dt=\frac \pi 2 \sqrt{\frac K N}.\qedhere
\end{equation*}
\end{proof}

%and by \eqref{eq:gamma interms of partial},
%\begin{equation*}
%\Gamma (x)=\sum_{j=1}^{d}(\partial_j x)^\ast \partial_j x.
%\end{equation*}
%Using the Kadison-Schwarz inequality and the , one obtains
%\begin{align*}
%P_{s}\Gamma(P_{t-s} a)
%=&\sum_{j=1}^{d}P_{s}\left( \left(\partial_j P_{t-s} a\right)^\ast \left(\partial_j P_{t-s} a\right)\right)\\
%\ge& \sum_{j=1}^{d} \left(P_s\partial_j P_{t-s} a\right)^\ast \left(P_s\partial_j P_{t-s} a\right)\\
%=&\sum_{j=1}^{d}e^{-2K_j (t-s)} \left(P_t\partial_j a\right)^\ast \left(P_t\partial_j  a\right)\\
%= &e^{-2K (t-s)}\sum_{j=1}^{d} \left(P_t\partial_j a\right)^\ast \left(P_t\partial_j  a\right).
%\end{align*}

%\begin{align*}
%P_s\Gamma(x)=\sum_{j=1}^{d}P_s\left((\partial_j x)^\ast \partial_j x\right)
%\ge &\sum_{j=1}^{d}(P_s\partial_j x)^\ast (P_s\partial_j x)\\
%= &\sum_{j=1}^{d}e^{2K_j s}(\partial_j P_s x)^\ast (\partial_j P_s x)\\
%\ge &e^{2K s}\sum_{j=1}^{d}(\partial_j P_s x)^\ast (\partial_j P_s x)
%=e^{2 K s}\Gamma (P_s x).
%\end{align*}
%Therefore 
%\begin{align*}
%\langle P_t(a^\ast a),\rho\rangle-\langle (P_t a)^\ast P_t a,\rho\rangle
%=\int_{0}^{t}\langle P_{s}\Gamma(P_{t-s} a),\rho\rangle ds
%&\ge \int_{0}^{t}e^{-2 K (t-s)} ds\sum_{j=1}^{d} \langle \left(P_t\partial_j a\right)^\ast \left(P_t\partial_j  a\right),\rho\rangle \\
%&=\frac{1-e^{-2Kt}}{2K}\sum_{j=1}^{d} \langle \left(P_t\partial_j a\right)^\ast \left(P_t\partial_j  a\right),\rho\rangle.
%\end{align*}

\subsection{\texorpdfstring{Complete $\GE(K,N)$}{Complete GE(K,N)}}
Now we turn to the complete version of $\GE(K,N)$.

\begin{definition}
	We say that a quantum Markov semigroup $(P_t)$ satisfies complete gradient estimate $\CGE(K,N)$ for $K\in \mathbb{R}$ and $N\in (0,\infty]$ if $(P_t\otimes\id_{M_n(\IC)})$ satisfies $\GE(K,N)$ for all $n\in\IN$.
\end{definition}

Similar to Proposition \ref{prop:intertwining implies CBE}, the $K$-intertwining condition is sufficient for $\CGE\colon$

\begin{proposition}\label{prop:intertwining implies CGE}
	Suppose that the generator $\cL$ of the quantum Markov semigroup $(P_t)$ admits the Lindblad form \eqref{eq:lindblad} with $d$ partial derivatives $\partial_1,\dots,\partial_d$. If $(P_t)$ satisfies the $K$-intertwining condition for $K\in\mathbb{R}$, then $(P_t)$ satisfies $\CGE(K,d)$.
\end{proposition}

Also, the complete gradient estimate $\CGE$ is tensor stable.
\begin{proposition}
	Consider two quantum Markov semigroups $(P_t^j)$ acting on $\M_j$, $j=1,2$. If $(P_t^j)$ satisfies $\CGE(K_j,N_j),j=1,2$, then the tensor product $(P_t^1\otimes P_t^2)$ over $\cM=\cM_1\otimes\cM_2$ satisfies $\CGE(K,N)$ with $K=\min\{K_1,K_2\}$ and $N=N_1+ N_2$.
\end{proposition}
\begin{proof}
	For each $j=1,2$, we denote by $\L_j$ the generator of $(P_t^j)$ and $\partial^j:\cM_j\to \hat{\cM}_j$ (to distinguish from partial derivatives $\partial_j$'s) the corresponding derivation operator so that $\L_j=(\partial^{j})^\dagger\partial^j$. 
	Denote $P_t=P_t^1\otimes P_t^2$. Then its generator is $\L=\partial^\dagger\partial$, where the derivation operator $\partial $ is given by
	\begin{equation*}
	\partial=(\partial^1\otimes \id,\id \otimes \partial^2).
	\end{equation*}
	Since $(P_t^j)$ satisfies $\CGE(K,N_j),j=1,2$, we have for any $a\in \hat{\cM}:=\otimes_j\hat{\cM}_j$ and $\rho\in \mathcal{S}_+(\cM)$ that
	\begin{align*}
	\norm{\partial P_t a}_\rho^2
	=&\|(\partial^1 \otimes \id)(P_t^1\otimes \id)(\id \otimes P_t^2)a\|_{\rho}^2+\|(\id\otimes\partial^2)(\id\otimes P_t^2)(P_t^1 \otimes \id)a\|_{\rho}^2\\
	\leq &e^{-2Kt}\left( \|(\partial^1 \otimes \id)(\id \otimes P_t^2)a\|_{(P_t^1\otimes \id)\rho}^2
	+\|(\id\otimes\partial^2)(P_t^1 \otimes \id)a\|_{(\id\otimes P_t^2)\rho}^2\right)\\
	&-\frac{1-e^{-2Kt}}{K}\left(\frac{1}{N_1}\abs{\langle (\L_1\otimes \id) P_t a,\rho\rangle}^2+\frac{1}{N_2}\abs{\langle (\id\otimes \L_2) P_t a,\rho\rangle}^2\right).
	\end{align*}
	
	As we have proven in \cite[Theorem 4.1]{WZ20}, for the first summand one has
	\begin{equation*}
	\|(\partial^1 \otimes \id)(\id \otimes P_t^2)a\|_{(P_t^1\otimes \id)\rho}^2
	+\|(\id\otimes\partial^2)(P_t^1 \otimes \id)a\|_{(\id\otimes P_t^2)\rho}^2\le \norm{\partial a}_{P_t \rho}^2.
	\end{equation*}
	As for the second summand, note that $\L=\L_1\otimes \id+\id\otimes \L_2$. So by Cauchy-Schwarz inequality,
	\begin{align*}
	\frac 1 N \abs{\langle\L P_t a,\rho\rangle}^2
	&=\frac{\abs{\langle (\L_1\otimes\id) P_t a,\rho\rangle+\langle (\id\otimes\L_2) P_t a,\rho\rangle}^2}{N_1+N_2}\\
	&\leq \frac{1}{N_1}\abs{\langle (\L_1\otimes \id) P_t a,\rho\rangle}^2+\frac{1}{N_2}\abs{\langle (\id\otimes \L_2) P_t a,\rho\rangle}^2.
	\end{align*}
	All combined, we obtain
	\begin{equation*}
	\norm{\partial P_t a}_\rho^2\le e^{-2Kt}\norm{\partial a}_{P_t \rho}^2-\frac {1-e^{-2Kt}}{K N} \abs{\langle\L P_t a,\rho\rangle}^2.\qedhere
	\end{equation*}
\end{proof}
\section{\texorpdfstring{Geodesic $(K,N)$-convexity of the (relative) entropy and relation to the gradient estimate $\GE(K,N)$}{Geodesic (K,N)-convexity of the entropy and relation to the gradient estimate GE(K,N)}}
\label{sec:geodesic convexity}

In the case of the logarithmic mean, the given quantum Markov semigroup is the gradient flow of the (relative) entropy with respect to the transport distance $\W$. In this case, the gradient estimate $\GE(K,\infty)$ is equivalent to geodesic $K$-convexity of the (relative) entropy with respect to $\W$, and several functional inequalities can be obtained using gradient flow techniques. 

Similarly, the gradient estimate $\GE(K,N)$ is equivalent to geodesic $(K,N)$-convexity of the (relative) entropy with respect to $\W$, a notion introduced by Erbar, Kuwada and Sturm \cite{EKS15}, and again, gradient flow techniques allow to deduce several dimensional functional inequalities from the abstract theory of $(K,N)$-convex functions on Riemannian manifolds.% In fact, we will show that this approach is not restricted to the (relative) entropy, but also covers a larger class of entropy-like functionals.

\subsection{\texorpdfstring{$(K,N)$-convexity for the (relative) entropy}{(K,N)-convexity for the (relative) entropy}}

Let $(M,g)$ be a Riemannian manifold and $K\in \IR$, $N\in (0,\infty]$. A function $S\in C^2(M)$ is called \emph{$(K,N)$-convex} if
\begin{equation*}
\Hess S(x)[v,v]-\frac 1 N g(\nabla S(x),v)^2\geq K g(v,v)
\end{equation*}
for all $x\in M$ and $v\in T_x M$.

With the function
\begin{equation*}
U_N\colon M\to \IR,\,U_N(x)=\exp\left(-\frac 1 N S(x)\right),
\end{equation*}
the $(K,N)$-convexity of $S$ can equivalently be characterized by
\begin{equation*}
\Hess U_N\leq -\frac K N U_N.
\end{equation*}
For $N=\infty$, one obtains the usual notion of $K$-convexity. Moreover, the notion of $(K,N)$-convexity is obviously monotone in the parameters $K$ and $N$ in the sense that if $S$ is $(K,N)$-convex, then $S$ is also $(K',N')$-convex for $K'\leq K$ and $N'\geq N$.

Our focus will be on the case when $F$ is the (relative) entropy and the Riemannian metric is the one introduced in \cite{CM17,CM20}, whose definition was recalled in Subsection \ref{subsec:Bonnet-Myers}.

If $F\colon \mathcal{S}_+(\M)\to\IR$ is smooth, its Frechét derivative can be written as
\begin{equation*}
d F(\rho)=\tau(x\,\cdot)
\end{equation*}
for a unique traceless self-adjoint $x\in\mathcal{M}$. This element $x$ shall be denoted by $DF(\rho)$. In particular, if $F(\rho)=\tau(\rho\log \rho)$, then $DF(\rho)=\log \rho+c$ for some $c\in\IR$.

By \cite[Theorem 7.5]{CM17}, the gradient of $F$ is given by (recall \eqref{eq:Riemannian operator} for $\K_\rho^\Lambda$)
\begin{equation}\label{eq:gradient_Wasserstein}
\nabla_{g^\Lambda} F(\rho)=\K_\rho^\Lambda DF(\rho).
\end{equation}
Of particular interest to us is the case when $F$ is the (relative) entropy, that is, the functional
\begin{equation*}
\Ent\colon \mathcal{S}_+(\M)\to (0,\infty),\,\Ent(\rho)=\tau(\rho\log \rho).
\end{equation*}
If we choose $\Lambda$ to be the logarithmic mean $\Lambda_{\log}$, then $\rho_t=P_t \rho$ satisfies the gradient flow equation
\begin{equation*}
\dot \rho_t=-\nabla_{g^\Lambda}\Ent(\rho_t)
\end{equation*}
for any $\rho\in \mathcal{S}_+(\M)$ \cite[Theorem 7.6]{CM17}. For this reason, we fix the operator mean $\Lambda$ to be the logarithmic mean in this section.

%\begin{lemma}
%If $F\colon \mathcal{S}_+(\mathcal{M})\to\IR$ is smooth, then
%\begin{equation*}
%\nabla_g F(\rho)=\K_\rho DF(\rho).
%\end{equation*}
%\end{lemma}
%\begin{proof}
%Let $\rho\colon (-\epsilon,\epsilon)\to \mathcal{S}_+(\M)$ be a smooth curve with $\rho_0=\rho$. By the definition of $g$ we have
%\begin{align*}
%\frac{d}{dt}\bigg|_{t=0}F(\rho_t)=g(\nabla_g F(\rho),\dot \rho_0)=\langle \K_\rho^{-1}\nabla_g F(\rho),\dot \rho_0).
%\end{align*}
%On the other hand,
%\begin{equation*}
%\frac{d}{dt}\bigg|_{t=0}F(\rho_t)=\tau\left(\frac{\delta F}{\delta\rho}(\rho) \dot \rho_0\right).
%\end{equation*}
%Thus
%\begin{equation*}
%\frac{\delta F}{\delta \rho}(\rho)=\K_\rho^{-1}\nabla_g F(\rho).\qedhere.
%\end{equation*}
%\end{proof}

To formulate the metric formulations of $(K,N)$-convexity, we need the following notation: For $\theta,\kappa\in\IR$ and $t\in [0,1]$ put
\begin{align}
\begin{split}\label{eq:defn of function sigma}
c_{\kappa}(\theta)&=\begin{cases}\cos(\sqrt{\kappa}\theta),&\text{if }\kappa\geq 0,\\\cosh(\sqrt{-\kappa}\theta),&\text{if }\kappa<0,\end{cases}\\
s_{\kappa}(\theta)&=\begin{cases}\kappa^{-1/2}\sin(\sqrt{\kappa}\theta),&\text{if }\kappa>0,\\
\theta,&\text{if }\kappa=0,\\(-\kappa)^{-1/2}\sinh(\sqrt{\kappa}\theta),&\text{if }\kappa<0,\end{cases}\\
\sigma_{\kappa}^{(t)}(\theta),&=
\begin{cases}
\frac{s_{\kappa}(t\theta)}{s_{\kappa}(\theta)}, & \kappa \theta^2\neq 0 \text{ and }\kappa\theta^2<\pi^2,\\
t,&\kappa \theta^2=0,\\
+\infty,&\kappa\theta^2\ge \pi^2.
\end{cases}
\end{split}
\end{align}

The following theorem is a quite direct consequence of the abstract theory of $(K,N)$-convex functions and the computation of the gradient and Hessian on $(\mathcal{S}_+(\M),g)$ carried out in \cite{CM17,CM20}. Nonetheless, it implies some interesting functional inequalities, as we shall see in the following subsection.
\begin{theorem}\label{thm:char_(K,N)-convexity}
Fix the logarithmic mean $\Lambda=\Lambda_{\log}$. Let $K\in\IR$ and $N\in (0,\infty]$. Further let 
\begin{equation*}
U_N(\rho)=\exp\left(-\frac 1 N \Ent(\rho)\right).
\end{equation*}
The the following assertions are equivalent:
\begin{enumerate}[(a)]
\item The (relative) entropy $\Ent$ is $(K,N)$-convex on $(\mathcal{S}_+(\M), g^{\Lambda})$.
\item For all $\rho,\nu\in \mathcal{S}_+(\M)$, the following Evolution Variational Inequality holds for all $t\ge 0$:
\begin{align*}
\frac{d^+}{dt}s_{K/N}\left(\frac 1 2\W(P_t\rho,\nu)\right)^2	+Ks_{K/N}\left(\frac 1 2\W(P_t\rho,\nu)\right)^2	\le\frac{N}{2}\left(1-\frac{U_N(\nu)}{U_N(P_t\rho)}\right).\tag{EVI$_{K,N}$}
\end{align*}
\item For any constant speed geodesic $(\rho_t)_{t\in [0,1]}$ in $\mathcal{S}_+(\M)$ one has
\begin{equation*}
U_N(\rho_t)\ge \sigma^{(1-t)}_{K/N}(\W(\rho_0,\rho_1))U_N(\rho_0)+\sigma^{(t)}_{K/N}(\W(\rho_0,\rho_1))U_N(\rho_1),~~t\in [0,1].
\end{equation*}
\item The semigroup $(P_t)$ satisfies $\GE(K,N)$.
	\end{enumerate}
\end{theorem}
\begin{proof}
(a) $\iff$ (b)$\iff$(c): These equivalences follow from abstract theory of $(K,N)$-convex functionals on Riemannian manifolds \cite[Lemmas 2.2, 2.4]{EKS15}.

(a)$\iff$(d): With the identification of the gradient from (\ref{eq:gradient_Wasserstein}) and the Hessian from \cite[Proposition 7.16]{CM20}, one sees that the defining inequality of the $(K,N)$-convexity of $D$ coincides with the equivalent formulation of $\GE(K,N)$ given in Proposition \ref{prop:equiv_GE_CD}.
\end{proof}

\subsection{Dimension-dependent functional inequalities}

Let us first collect some consequences of $(K,N)$ convexity that were already observed in \cite{EKS15}, adapted to our setting. Recall that $\Ent(\rho)=\tau(\rho\log \rho)$. We use the notation
\begin{equation*}
\I(\rho)=\tau((\L\rho)\log \rho)
\end{equation*}
for the Fisher information.

It satisfies the de Bruijn identity 
\begin{equation*}
\frac{d}{dt}\Ent(P_t \rho)=-\I(P_t \rho).
\end{equation*}

The following inequalities (b) (c) and (d) are finite-dimensional versions of the HWI-inequality, modified log-Sobolev inequality (MLSI) and Talagrand inequality, respectively. The infinite-dimensional results (i.e. $N=\infty$) were obtained in \cite{CM17,CM20,DR20}.
\begin{proposition}
Fix the logarithmic mean $\Lambda=\Lambda_{\log}$. Let $K\in \IR$ and $N>0$. If $(P_t)$ satisfies $\GE(K,N)$, then the following functional inequalities hold:
	\begin{enumerate}[(a)]
		\item $\W$-expansion bound:
		\begin{align*}
		&s_{K/N}\left(\frac 1 2\W(P_t\rho_0,P_s\rho_1)\right)^2\\
		&\qquad\leq e^{-K(s+t)}s_{K/N}\left(\frac 1 2\W(\rho_0,\rho_1)\right)^2+\frac{N}{K}\left(1-e^{-K(s+t)}\right)\frac{(\sqrt{t}-\sqrt{s})^2}{2(s+t)}
		\end{align*}
		for $\rho_0,\rho_1\in \mathcal{S}_+(\M)$ and $s,t\geq 0$.
		\item $N$-HWI inequality:
		\begin{equation*}
		\frac{U_N(\rho_1)}{U_N(\rho_0)}\leq c_{K/N}(\W(\rho_0,\rho_1))+\frac 1 N s_{K/N}(\W(\rho_0,\rho_1))\sqrt{\I(\rho_0)},
		\end{equation*}
		for $\rho_0,\rho_1\in \mathcal{S}_+(\M)$ and $s,t\geq 0$.
\end{enumerate}
If $K>0$, then additionally the following functional inequalities hold:
\begin{enumerate}[(a)]
\setcounter{enumi}{2}
		\item $N$-MLSI:
		\begin{equation*}
		KN\left(U_N(\rho)^{-2}-1\right)\leq \I(\rho),
		\end{equation*}
		for $\rho\in \mathcal{S}_+(\M)$.
		\item $N$-Talagrand inequality:
		\begin{equation*}
		\Ent(\rho)\ge -N\log\cos\left(\sqrt{\frac{K}{N}}\W(\rho,\un)\right),
		\end{equation*}
		for $\rho\in\mathcal{S}_+(\M)$.
	\end{enumerate}
\end{proposition}
\begin{proof}
The proofs of Theorems 2.19, 3.28 and Corollaries 3.29, 3.31 from \cite{EKS15} can be easily adapted to our setting.
\end{proof}

\subsection{Concavity of entropy power}
Let us now move on to the concavity of entropy power:
$$t\mapsto U_N(P_t\rho)^2=\exp\left(-\frac{2}{N}\Ent(P_t\rho)\right).$$
For the heat semigroup on $\IR^n$, the concavity of entropy power along the heat flow was first proved by Costa in \cite{Costa85entropypower}. In \cite{Villani00concavity} Villani gave a short proof and remarked that this can be proved using $\Gamma_2$-calculus. Recently Li and Li \cite{LL20entropypower} considered this problem on the Riemannian manifold under the curvature-dimension condition CD($K,N$). % \textcolor{red}{For more related work see ?...maybe we don't need more references here}
Here we show that the geodesic concavity of the entropy power follows from the $(K,N)$-convexity of the entropy.

\begin{theorem}\label{thm:entropy power}
Let $K\in \IR$ and $N>0$. If $(P_t)$ satisfies $\GE(K,N)$ for logarithmic mean, then
\begin{equation*}
\frac{d^2}{dt^2}U_N(P_t\rho)^2\le -2K \frac{d}{dt} U_N(P_t \rho)^2,~~t\ge 0.
\end{equation*}
In particular, if $K\ge 0$, then $\frac{d^2}{dt^2}U_N(P_t \rho)^2\leq 0$. This implies the concavity of the entropy power $t\mapsto U_N(P_t \rho)^2$.
\end{theorem}
\begin{proof}
Let $\rho_t=P_t \rho$. Since $\Ent$ is $(K,N)$-convex by Theorem \ref{thm:char_(K,N)-convexity} and $(P_t)$ is a gradient flow of $\Ent$ in $(\mathcal{S}_+(\M),g)$ by our choice of the operator mean, we have
\begin{align*}
\frac{d^2}{dt^2}U_N(\rho_t)^2&=\frac{d}{dt}\left(-\frac 2N \langle\nabla_g \Ent(\rho_t),\dot \rho_t\rangle U_N(x_t)\right)\\
&=\frac{d}{dt}\left(\frac 2 N\langle \dot \rho_t,\dot \rho_t\rangle U_N(\rho_t)^2\right)\\
&=\left(\frac 4 N \langle \dot \rho_t,\nabla_{\dot \rho_t}\dot \rho_t\rangle +\frac 4{N^2} \langle \dot \rho_t,\dot \rho_t\rangle^2\right)U_N(\rho_t)^2\\
&=\frac 4 N\left(-\Hess \Ent(\rho_t)[\dot \rho_t,\dot \rho_t]+\frac 1 N \langle \nabla \Ent(\rho_t),\dot \rho_t\rangle^2\right)U_N(\rho_t)^2\\
&\leq -\frac{4K}{N}\langle \dot \rho_t,\dot \rho_t\rangle U_N(\rho_t)^2\\
&=-2K\frac{d}{dt}U_N(\rho_t)^2.\qedhere
\end{align*}
\end{proof}
\begin{remark}
The same proof applies abstractly whenever $F$ is a $(K,N)$-convex functional on a Riemannian manifold and $(\rho_t)$ is a gradient flow curve of $F$.
\end{remark}

The following proof is closer to the spirit of Villani.

\begin{proof}[Another proof of Theorem \ref{thm:entropy power}]
	Put $\varphi(t):=-\Ent(\rho_t)=-\tau(\rho_t\log\rho_t)$ with $\rho_t=P_t\rho$. Recall the chain rule
	$$\partial \rho=\hat{\rho}\partial \log\rho.$$
	Thus 
	\begin{equation}\label{eq:varphi'}
	\varphi'(t)
	%=\E(\rho_t,f'(\rho_t))
	=\langle \cL\rho_t,\log\rho_t\rangle
	=\langle \hat{\rho_t}\partial \log\rho_t),\partial \log\rho_t\rangle.
	\end{equation}
	This allows to give two forms of $\varphi''$:
	\begin{equation}\label{eq:1st form of varphi''}
	\varphi''(t)=\frac{d}{dt}\langle \cL\rho_t, \log\rho_t\rangle
	=\langle \cL \rho_t, \frac{d}{dt} \log\rho_t\rangle-\langle \cL\rho_t,\cL  \log\rho_t\rangle=:\mathrm{I},
	\end{equation}
	and 
	\begin{align}
	\varphi''(t)&=\frac{d}{dt}\langle\widehat{\rho_{t}}\partial  \log\rho_t,\partial  \log\rho_t\rangle \nonumber\\ 
	&=2 \langle \widehat{\rho_t}\partial  \log\rho_t,\partial \frac{d}{dr}\big|_{r=t}\log\rho_r\rangle
	+\langle \frac{d}{dr}\big|_{r=t}\widehat{\rho_{r}}\partial  \log\rho_t,\partial  \log\rho_t\rangle \nonumber\\
	&=2\langle \cL \rho_t, \frac{d}{dt} \log\rho_t\rangle
	+\langle \frac{d}{dr}\big|_{r=t}\widehat{\rho_{r}}\partial  \log\rho_t,\partial  \log\rho_t\rangle=:\mathrm{II}. \label{eq:2nd form of varphi''}
	\end{align}	
	From \eqref{eq:1st form of varphi''} and \eqref{eq:2nd form of varphi''} we deduce that 
	\begin{equation}\label{eq:3rd form of varphi''}
	\varphi''(t)=2\mathrm{I}-\mathrm{II}
	=-2\langle \cL\rho_t,\cL \log\rho_t\rangle-\langle \frac{d}{dr}\big|_{r=t}\widehat{\rho_{r}}\partial \log\rho_t,\partial \log\rho_t\rangle.
	\end{equation}
	Since $(P_t)$ satisfies $\GE(K,N)$ we have by Proposition \ref{prop:equiv_GE_CD} that% (applying \eqref{ineq:GE(K,N) equiv} to $(a,\rho)=(f'(\rho_t),\rho_t)$)
	\begin{equation*}
	\langle \hat{\rho_t}\partial \cL \log\rho_t,\partial \log\rho_t\rangle
	+\frac{1}{2}\langle \frac{d}{dr}\big|_{r=t}\widehat{\rho_r}\partial \log\rho_t,\partial \log\rho_t\rangle 
	\ge K \norm{\partial \log\rho_t}_{\rho}^2+\frac{1}{N}\abs{\E( \log\rho_t,\rho_t)}^2,
	\end{equation*}
	that is,
	\begin{equation*}
	2\langle \cL \log\rho_t,\cL \rho_t\rangle
	+\langle \frac{d}{dr}\big|_{r=t}\widehat{\rho_r}\partial \log\rho_t,\partial \log\rho_t\rangle 
	\ge 2K \norm{\partial \log\rho_t}_{\rho}^2+\frac{2}{N}\abs{\E( \log\rho_t,\rho_t)}^2.
	\end{equation*}
	This, together with \eqref{eq:varphi'} and \eqref{eq:3rd form of varphi''}, yields
	\begin{equation}\label{ineq:varphi' varphi''}
	\varphi''(t)
	\le -2K\|\partial \log\rho_t\|_{\rho_t}^2-\frac{2}{N}|\E(\log\rho_t,\rho_t)|^2
	=-2K\varphi'(t)-\frac{2}{N}\varphi'(t)^2.
	\end{equation}
	A direct computation gives
	\begin{equation*}
	\frac{d}{dt}U_N(P_t\rho)^2=\frac{2}{N}U_N(P_t\rho)^2\varphi'(t),
	\end{equation*}
	and 
	\begin{equation*}
	\frac{d^2}{dt^2}U_N(P_t\rho)^2=\frac{2}{N}U_N(P_t\rho)^2\left(\frac{2}{N}\varphi'(t)^2+\varphi''(t)\right).
	\end{equation*}
 So by \eqref{ineq:varphi' varphi''} we get
	\begin{equation*}
	\frac{d^2}{dt^2}U_N(P_t\rho)^2
	\le -\frac{4K}{N}U_N(P_t\rho)^2\varphi'(t)
	=-2K\frac{d}{dt}U_N(P_t\rho)^2.\qedhere
	\end{equation*}
\end{proof}

\begin{remark}
	Here we used the fact that $\mathrm{I}=\mathrm{II},$ or equivalently,
	\begin{equation*}%\label{eq:chain rule-villani}
	\langle \cL \rho_t, \frac{d}{dt}\log\rho_t\rangle
	+\langle \cL\rho_t,\cL \log\rho_t\rangle
	+\langle \frac{d}{dr}\big|_{r=t}\widehat{\rho_{r}}\partial \log\rho_t,\partial \log\rho_t\rangle=0.
	\end{equation*}
	If we consider the heat semigroup $P_t=e^{t\Delta}$ on $\mathbb{R}^n$, then this follows from the elementary identity 
	\begin{equation*}
	\frac{\Delta f}{f}=\Delta(\log f)+|\nabla(\log f)|^2,
	\end{equation*}
	as used in Villani's proof \cite{Villani00concavity}.
\end{remark}

\section{Examples}
\label{sec:examples}

In this section we present several classes of examples of quantum Markov semigroups satisfying $\BE(K,N)$ and $\GE(K,N)$. The verification of these examples relies crucially on the criteria from Proposition \ref{prop:BE_intertwining} and Theorem \ref{thm:CD under intertwining}.

\subsection{Schur multipliers over matrix algebras}
A Schur multiplier $A$ over the $n\times n$ matrix algebra $M_n(\IC)$ is a linear map of the form:
$$Ae_{ij}:=a_{ij}e_{ij},$$
where $a_{ij}\in \IC$ and $\{e_{ij}\}_{i,j=1}^n$ are the matrix units. By Schoenberg's theorem (see for example \cite[Appendix D]{BO08}), 
$$P_t[x_{ij}]=e^{-tA} [x_{ij}]=[e^{-t a_{ij}}x_{ij}],~~t\ge 0,$$
defines a symmetric quantum Markov semigroup over $M_n(\IC)$ if and only if 
\begin{enumerate}[(a)]
	\item $a_{ii}=0$ for all $1\le i\le n$,
	\item $a_{ij}=a_{ji}\ge 0$ for all $1\le i,j\le n$,
	\item $[a_{ij}]$ is \emph{conditionally negative definite}: 
	$$\sum_{i,j=1}^{n}\overline{\alpha_i}\alpha_ja_{ij}\le 0,$$
	whenever $\alpha_1,\dots,\alpha_n$ are complex numbers such that $\sum_{j=1}^{n}\alpha_j=0$.
\end{enumerate}
If this is the case, then there exists a real Hilbert space $H$ and elements $a(j)\in H$, $1\le j\le n$, such that
$$a_{ij}=\|a(i)-a(j)\|^2,~~1\le i,j\le n.$$
Suppose that $(e_k)_{1\le k\le d}$ is an orthonormal basis of $H$. Define for each $1\le k\le d$ 
$$v_k:=\sum_{j=1}^{n}\langle a(j),e_k\rangle e_{jj}\in M_n(\IC).$$
Then for any $1\le i,j\le n$:
\begin{equation*}
[v_k,e_{ij}]=v_k e_{ij}- e_{ij}v_k=\langle a(i)-a(j),e_k\rangle e_{ij},
\end{equation*}
and 
\begin{equation*}
[v_k,[v_k,e_{ij}]]=|\langle a(i)-a(j),e_k\rangle|^2 e_{ij}.
\end{equation*}
By the choice of $(e_k)$, we have
\begin{equation*}
\sum_{k=1}^{d}[v_k,[v_k,e_{ij}]]=\|a(i)-a(j)\|^2 e_{ij}=a_{ij}e_{ij}.
\end{equation*}
Therefore, 
$$A=\sum_{k=1}[v_k,[v_k,\cdot]],$$
and it is easy to see that $[v_k,A\cdot]=A[v_k,\cdot]$ for each $k$. So by Propositions \ref{prop:intertwining implies CBE} and \ref{prop:intertwining implies CGE} we have $\CBE(0,d)$ and $\CGE(0,d)$ for any operator mean.

\subsection{Herz-Schur multipliers over group algebras}
Let $G$ be a finite group. Suppose that $\lambda$ is the left-regular representation, i.e. for $g\in G$,
\begin{equation*}
\lambda_g\colon \ell_2(G)\to\ell_2(G),\,\lambda_g \1_h=\1_{gh},
\end{equation*}
where $\1_{h}$ is the delta function at $h$. The group algebra of $G$ is then the (complex) linear span of $\{\lambda_g\mid g\in G\}$, denoted by $\IC[G]$. It carries a canonical tracial state $\tau$ given by $\tau(x)=\langle x\1_e,\1_e\rangle$, where $e$ is the unit element of $G$.

We say that $\ell\colon G\to [0,\infty)$ is a \emph{conditionally negative definite length function} if $\ell(e)=0$, $\ell(g^{-1})=\ell(g)$ for all $g\in G$ and
\begin{equation*}
\sum_{g\in G}\bar \alpha_g \alpha_h \ell(g^{-1}h)\leq 0
\end{equation*}
whenever $\alpha_g$, $g\in G$, are complex numbers such that $\sum_{g\in G}\alpha_g=0$. By Schoenberg's theorem (see for example \cite[Appendix D]{BO08}), there exists a 1-cocycle $(H,\pi, b)$ consisting of a real Hilbert space $H$ of dimension $\dim H\leq \abs{G}-1$, a unitary representation $\pi\colon G\to B(H)$ and a map $b\colon G\to H$ satisfying the cocycle condition
\begin{equation*}
b(gh)=b(g)+\pi(g)b(h)
\end{equation*}
for $g,h\in G$ such that $\ell(g)=\norm{b(g)}^2$.

Every conditionally negative definite length function $\ell$ on $G$ induces a $\tau$-symmetric quantum Markov semigroup $(P_t)$ on $\IC[G]$ characterized by $P_t\lambda_g=e^{-t\ell(g)}\lambda_g$ for $g\in G$. Let $e_1,\dots,e_d$ be an orthonormal basis of $H$. As argued in \cite{WZ20} (or similar to the Schur multipliers case), the generator $\L$ of $(P_t)$ can be written as
\begin{equation*}
\L =\sum_{j=1}^d [v_j,[v_j,\cdot\,]]
\end{equation*}
with $d=\dim H$ and
\begin{equation*}
v_j\colon \ell_2(G)\to \ell_2(G),\,v_j\1_h=\langle b(h),e_j\rangle \1_h.
\end{equation*}
The operators $v_j$ are not contained in $\IC[G]$ in general, but one can extend $\L$ to a linear operator on $B(\ell_2(G))$ by the same formula, and a direct computation shows $[v_j,\L\,\cdot\,]=\L[v_j,\cdot\,]$. By Propositions \ref{prop:intertwining implies CBE} and \ref{prop:intertwining implies CGE}, $(P_t)$ satisfies $\CBE(0,d)$ and $\CGE(0,d)$ for any operator mean. 

\begin{example}
	The cyclic group $\mathbb{Z}_n=\{0,1,\dots,n-1\}$; see \cite[Example 5.9]{JZ15a} or \cite[Example 5.7]{WZ20}: Its group (von Neumann) algebra is spanned by $\lambda_k,0\le k\le n-1$. One can embed $\mathbb{Z}_n$ to $\mathbb{Z}_{2n}$, so let us assume that $n$ is even. The word length of $k\in\mathbb{Z}_n$ is given by $\ell(k)=\min\{k,n-k\}$. The associated 1-cocycle can be chosen with $H=\mathbb{R}^{\frac{n}{2}}$ and $b\colon\mathbb{Z}_n\to \mathbb{R}^{\frac{n}{2}}$ via
	\begin{equation*}
	b(k)=\begin{cases}
	0,&k=0,\\
	\sum_{j=1}^{k}e_j,&1\le k\le \frac{n}{2},\\
	\sum_{j=k-\frac{n}{2}+1}^{\frac{n}{2}}e_j,&\frac{n}{2}+1\le k\le n-1,
	\end{cases}
	\end{equation*}
	where $(e_j)_{1\le j\le \frac{n}{2}}$ is an orthonormal basis of $\mathbb{R}^{\frac{n}{2}}$. Thus the quantum Markov semigroup associated with $\ell$ satisfies $\CBE(0,n/2)$ and $\CGE(0,n/2)$ for any operator mean. 
\end{example}

\begin{example}
	The symmetric group $S_n$; see \cite[Example 5.8]{WZ20}: Let $\ell$ be the length function induced by the (non-normalized) Hamming metric, that is, $\ell(\sigma)=\#\{j : \sigma(j)\neq j\}$. Let $A_\sigma\in M_n(\IR)$ be the permutation matrix associated with $\sigma$, i.e., $A_\sigma \delta_j =\delta_{\sigma(j)}$. Then the associated cocycle is given by $H=L^2(M_n(\IR),\frac 1 2 \mathrm{tr})$, $b(\sigma)=A_\sigma-1$ and  $\pi(\sigma)=A_\sigma$. Thus the quantum Markov semigroup associated with $\ell$ satisfies $\CBE(0,d)$ and $\CGE(0,d)$ for any operator mean with $d=\min\{|S_n|-1,n^2\}$.
\end{example}

\subsection{Depolarizing Semigroup}

Let $\tau$ be the normalized trace on $\M=M_d(\IC)$. The depolarizing semigroup $M_d(\IC)$ is defined by
\begin{equation*}
P_t\colon M_d(\IC)\to M_d(\IC),\,P_t a=e^{-t}a+(1-e^{-t})\tau(a)\un.
\end{equation*}
Its generator is given by $\L a=a-\tau(a)\un$. We will show that $(P_t)$ satisfies $\BE(1/2,2d)$ and $\GE(1/2,2d)$ for any operator mean $\Lambda$.

Recall that the generator admits a Lindblad form:
$$\L=\sum_{j\in \J}\partial_j^\dagger\partial_j.$$
Since $\partial_j(\un)=0$, we have $\partial_j P_t=e^{-t}\partial_j.$ Then
\begin{equation*}
\norm{\partial P_t a}_\rho^2=e^{-2t}\norm{\partial a}_{\rho}^2.
\end{equation*}
Fix $a\in \M$ and $\rho\in\M_+$ with $\tau(\rho)=1$. By positive homogeneity (Lemma \ref{lem:mean} (a)), concavity (Lemma \ref{lem:mean} (b)) and the definition of operator mean $\Lambda$, we get
\begin{align*}
\|\partial_j a\|^2_{P_t\rho}
= &\langle \Lambda(e^{-t}L(\rho)+(1-e^{-t})L(\un),e^{-t}R(\rho)+(1-e^{-t})R(\un))\partial_j a,\partial_j a\rangle\\
\geq &e^{-t}\langle \Lambda(L(\rho),R(\rho)) \partial_j a,\partial_j a\rangle+(1-e^{-t})\langle \partial_j a,\partial_j a\rangle.
\end{align*}
So 
\begin{equation*}
\|\partial a\|^2_{P_t\rho}\ge e^{-t}\|\partial a\|_{\rho}^2+(1-e^{-t})\langle \L a, a\rangle.
\end{equation*}
Note that $\L^2=\L$ and $P_t\rho\leq  d \un$. All combined, and using Cauchy-Schwarz inequality for $\tau(\rho\cdot)$, we obtain
\begin{align*}
\norm{\partial a}^2_{P_t\rho}&\geq e^{-t}\norm{\partial a}_\rho^2+(1-e^{-t})\langle \L a,\L a\rangle\\
&\geq e^{-t}\norm{\partial a}_\rho^2+\frac{1-e^{-t}}{d}\tau(\abs{\L a}^2P_t\rho)\\
&\geq e^{-t}\norm{\partial a}_\rho^2+\frac{1-e^{-t}}{d}\abs{\tau((\L a)P_t\rho)}^2\\
&=e^{t}\norm{\partial P_t a}_\rho^2+\frac{1-e^{-t}}{d}\abs{\E(a,P_t\rho)}^2,
\end{align*}
or equivalently,
$$\norm{\partial P_t a}_\rho^2\le e^{-t}\norm{\partial a}^2_{P_t\rho}-f(t)\abs{\E(a,P_t\rho)}^2.$$
Here $f(t)=(e^{-t}-e^{-2t})/d$, and it is easy to see $f(0)=0$ and $f'(0)=1/d$. By Remark \ref{rk:implicit constant_GE}, $(P_t)$ satisfies $\GE(1/2,2d)$. 

Choosing $\Lambda$ as the left trivial mean, we actually proved (without using Cauchy-Schwarz inequality):
\begin{equation*}
\tau(P_t(|\partial a|^2)\rho)
\geq e^{-t}\tau(|\partial a|^2\rho)+\frac{1-e^{-t}}{d}\tau(|\L a|^2P_t\rho).
\end{equation*}
Since both sides agree at $t=0$, we obtain by taking derivative at $t=0$ that 
\begin{equation*}
\Gamma_2(a)\ge \frac{1}{2}\Gamma(a)+\frac{1}{2d}|\L a|^2.
\end{equation*}
So $(P_t)$ satisfies $\BE(1/2,2d)$.

\section{Curvature-dimension conditions without assuming tracial symmetry}\label{sec:conclusion}

In plenty of applications one encounters quantum Markov semigroups that are not necessarily tracially symmetric, but only satisfy the detailed balance condition $\sigma$-DBC (with $\sigma\neq \un$) we mentioned in Section \ref{sec:QMS}. Many of the results from this article still apply in this case, with one important caveat, as we will discuss here.

The definition of the Bakry--Émery gradient estimate $\BE(K,N)$ makes sense for arbitrary quantum Markov semigroups on matrix algebras without any change, and all the consequences we proved stay valid in this more general setting with essentially the same proofs.

The gradient estimate $\GE(K,N)$ relies on the Lindblad form of the generator of the semigroup. By Alicki's theorem, a similar Lindblad form exists for generators of quantum Markov semigroups satisfying the $\sigma$-DBC, and the norms $\norm{\xi}_\rho$ have been defined in this setting in \cite{CM17,CM20} -- in fact, instead of a single operator mean one can choose a family of operator connections that depends on the index $j$. With this norm, one can formulate $\GE(K,N)$ as %({\color{red}no $\sigma$ involved in $|\cdot|^2$?})
\begin{equation*}
\norm{\partial P_t a}_\rho^2\leq e^{-2Kt}\norm{\partial a}^2_{P_t^\dagger \rho}-\frac{1-e^{-2Kt}}{KN}\abs{\tau((\L P_t a)\rho)}^2,
\end{equation*}
where one now has to distinguish between $P_t$ and $P_t^\dagger$ because of the lack of tracial symmetry.

The connection between a generalization of the metric $\W$, the semigroup $(P_t)$ and the relative entropy still remains true in this more general setting with an appropriate modification of the definition of $\W$ \cite{CM17,CM20}, so that the identification of $\GE(K,N)$ with the $(K,N)$-convexity condition for an entropy functional from Theorem \ref{thm:char_(K,N)-convexity} along with its applications also has an appropriate analog for quantum Markov semigroups satisfying the $\sigma$-DBC.

However, the criteria from Proposition \ref{prop:BE_intertwining} and Theorem \ref{thm:CD under intertwining}, which actually allow us to verify $\BE(K,N)$ and $\GE(K,N)$ in concrete examples, rely crucially on the Lindblad form of generators of tracially symmetric quantum Markov semigroups and do not immediately carry over to the $\sigma$-detailed balance case. Thus the question of proving $\BE(K,N)$ and $\GE(K,N)$ for not necessarily tracially symmetric quantum Markov semigroups remains open, so its usefulness in this case is still to be proven.

\bibliography{References}
\bibliographystyle{alpha}
\end{document}